\documentclass[11pt,a4paper]{article}

\usepackage{booktabs}
\usepackage{amsfonts}
\usepackage{amsmath,amsthm,amsopn}
\usepackage{graphicx}
\usepackage{algorithmic}
\usepackage[ruled]{algorithm2e}
\SetAlFnt{\small}
\SetAlCapFnt{\small}
\SetAlCapNameFnt{\small}
\SetAlCapHSkip{0pt}
\IncMargin{-\parindent}
\usepackage{tikz,pgfplots}
\usetikzlibrary{math}
\usepackage{xspace}
\usepackage{pdfcomment}
\usepackage[left=2.5cm,right=2.5cm,top=3cm,bottom=3cm]{geometry}
\usepackage{cleveref}
\Crefname{ALC@unique}{Line}{Lines}

\newcounter{myalg}
\AtBeginEnvironment{algorithmic}{\refstepcounter{myalg}}
\makeatletter
\@addtoreset{ALC@unique}{myalg}
\makeatother

\newenvironment{acknowledgements}{\paragraph{Acknowledgements}}{}
\usepackage{amsopn}

\newtheorem{theorem}{Theorem}

\newtheorem{definition}{Definition}
\newtheorem{lemma}{Lemma}

\newcommand{\mybar}[1]{\makebox[0pt]{$\phantom{#1}\overline{\phantom{#1}}$}#1}

\newcommand{\yambo}{Yambo\xspace}

\def\ub{\alpha}
\def\lb{\beta}

\def\cri{\textsf{CrI$_3$}\xspace}
\def\dna{\textsf{DNA}\xspace}

\def\MX{M_1}
\def\MY{M_2}
\def\oMX{\overline{M}_1}
\def\oMY{\overline{M}_2}
\def\QX{Q_1}
\def\QY{Q_2}
\def\oQX{\overline{Q}_1}
\def\oQY{\overline{Q}_2}
\def\WX{W_1}
\def\WY{W_2}
\def\oWY{\overline{W}_2}
\def\oWX{\overline{W}_1}
\def\XX{X_1}
\def\XY{X_2}
\def\oXX{\overline{X}_1}
\def\oXY{\overline{X}_2}

\def\RX{\mathcal{R}_1}
\def\RY{\mathcal{R}_2}
\def\oRX{\overline{\mathcal{R}}_1}
\def\oRY{\overline{\mathcal{R}}_2}

\def\VX{V_1}
\def\VY{V_2}
\def\oVX{\overline{V}_1}
\def\oVY{\overline{V}_2}
\def\LX{L_1}
\def\LY{L_2}
\def\oLX{\overline{L}_1}
\def\oLY{\overline{L}_2}
\def\NX{N_1}
\def\NY{N_2}
\def\oNX{\overline{N}_1}
\def\oNY{\overline{N}_2}
\def\SX{S_1}
\def\SY{S_2}
\def\oSX{\overline{S}_1}
\def\oSY{\overline{S}_2}

\def\SignGen{\Omega}
\def\SignIn{\Omega}
\def\SignOut{\tilde{\Omega}}
\def\SignBSE{\Sigma_m}
\def\SignOutBSE{\Sigma_k}
\def\sizeL{l}
\def\SignOutBSEL{\Sigma_{\sizeL}}

\def\SVDVal{\Upsilon}
\def\SVDV{V}
\def\SVDU{U}

\newcommand{\YnX}[1]{Y_1^{(#1)}}
\newcommand{\YnY}[1]{Y_2^{(#1)}}
\def\YYX{Y_1^{(n)}}
\def\YYY{Y_2^{(n)}}
\def\YYYaux{\check{Y}_2^{(n)}}
\def\oYYYaux{\overline{\check{Y}}_2^{(n)}}
\def\YZX{Y_1^{(0)}}
\def\YZY{Y_2^{(0)}}

\def\YVX{Y_1^{(n+1)}}
\def\YVY{Y_2^{(n+1)}}
\def\YWX{Y_1^{(n+2)}}
\def\YWY{Y_2^{(n+2)}}
\newcommand{\oYnX}[1]{\overline{Y}_1^{(#1)}}
\newcommand{\oYnY}[1]{\overline{Y}_2^{(#1)}}

\def\oYVX{\overline{Y}_1^{(n+1)}}
\def\oYVY{\overline{Y}_2^{(n+1)}}
\def\oYWX{\overline{Y}_1^{(n+2)}}
\def\oYWY{\overline{Y}_2^{(n+2)}}

\def\AX{\hat{X}}
\def\AQ{\hat{Q}}

\def\XX{X_1}
\def\XY{X_2}
\def\HX{\hat{X}}
\def\HXX{\hat{X}_1}
\def\HXY{\hat{X}_2}
\def\oHXX{\overline{\hat{X}}_1}
\def\oHXY{\overline{\hat{X}}_2}
\def\XXi{X_1^{(0)}}
\def\XYi{X_2^{(0)}}
\def\oXXi{\overline{X}_1^{(0)}}
\def\oXYi{\overline{X}_2^{(0)}}
\def\pHX{\tilde{X}}
\def\pHXX{\tilde{X}_1}
\def\pHXY{\tilde{X}_2}
\def\opHXX{\overline{\tilde{X}}_1}
\def\opHXY{\overline{\tilde{X}}_2}
\def\HQ{\hat{Q}}
\def\HQX{\hat{Q}_1}
\def\HQY{\hat{Q}_2}
\def\oHQX{\overline{\hat{Q}}_1}
\def\oHQY{\overline{\hat{Q}}_2}
\def\HQYaux{\check{Q}_2}
\def\oHQYaux{\overline{\check{Q}}_2}

\def\xallCrI{
1.8093308239718,
1.8101437976719,
1.9108518408822,
1.9126169486030,
3.1780490363699,
3.3782493063568,
3.3896047788142,
3.3901153941886,
3.3907928996599,
3.3943191630075,
3.3972046845099,
3.3973029175728,
3.3978170255284,
3.3984989112849,
3.3997640535358,
3.4050003551214,
3.8236966984244,
3.8307332888721,
3.8307377775414,
4.1507999233448,
4.1553790454539,
4.1600277964948,
4.1600484969790,
4.1649704068136,
4.1650533627705,
5.0108105344171,
6.5745247140083,
6.5745539982955,
6.5824329668307,
6.6093632560589,
6.6734415492003,
6.6743702700476,
8.5744253646747,
8.5758704814152,
8.5817615319823,
8.5818283974056,
8.5834669042292,
8.5867212544708,
8.5868024233719,
8.5916315096763,
8.5917144628525,
8.5924160222603,
8.5924249000937,
8.9223413566582,
8.9308676747290,
8.9313040688680,
8.9440906634177,
8.9441951779670,
8.9450294545535,
9.0409841143640,
9.1300301532293,
9.1507539964050,
9.1654446111867,
9.1655448787211,
9.1744820770881,
9.1750314390949,
10.5525586457303,
10.5971005765948,
10.5974982607336,
10.6017609381521,
10.6199802319903,
10.6207258634883,
10.6235508960380,
10.6364617282203,
10.6365209859263,
10.6370499811915,
10.6371106542139,
10.6373456088528,
10.6375292924500,
10.6375810697478,
10.6377334684910,
10.6435313223002,
10.6459168153978,
10.6459549596853,
10.6460021512840,
10.6460650771052,
10.6509821784042,
10.6681922246641,
10.6693966750490,
11.0052950631714,
11.1288343489389,
11.1344900966303,
11.1345768569469,
11.1429291613987,
11.1434449441485,
11.5487410816536,
11.5806045292339,
11.5806303738364,
11.6491235777964,
11.6732721044885,
11.6884766387602,
11.6885518586822,
11.6988203776589,
11.6988289186547,
11.9624368008692,
12.0411302301929,
12.0413660742009,
12.2148290388759,
12.2191093828284,
12.2191345145876
}

\def\xincCrI{
0,
0.000812973700099917,
0.1007080432103,
0.00176510772080007,
1.2654320877669,
0.2002002699869,
0.0113554724573999,
0.000510615374400114,
0.000677505471299966,
0.00352626334759965,
0.0028855215024004,
9.82330628995776E-05,
0.000514107955600363,
0.000681885756499856,
0.00126514225090002,
0.00523630158560007,
0.418696343303,
0.00703659044769989,
4.48866930025105E-06,
0.3200621458034,
0.00457912210910028,
0.00464875104090012,
2.0700484199665E-05,
0.00492190983460006,
8.29559569002214E-05,
0.8457571716466,
1.5637141795912,
2.92842872005039E-05,
0.00787896853519943,
0.0269302892282006,
0.0640782931413995,
0.000928720847300113,
1.9000550946271,
0.00144511674049852,
0.00589105056710082,
6.68654232995891E-05,
0.00163850682359978,
0.00325435024160115,
8.11689010991046E-05,
0.00482908630440093,
8.29531762001068E-05,
0.000701559407799479,
8.87783339997839E-06,
0.329916456564499,
0.00852631807080151,
0.000436394138999319,
0.0127865945497003,
0.000104514549299495,
0.000834276586500593,
0.0959546598105003,
0.0890460388652983,
0.0207238431757002,
0.014690614781701,
0.000100267534399734,
0.00893719836700058,
0.000549362006799115,
1.3775272066354,
0.044541930864499,
0.000397684138800258,
0.00426267741850062,
0.0182192938382002,
0.000745631497999,
0.00282503254969946,
0.0129108321823015,
5.92577059990163E-05,
0.000528995265201004,
6.06730223999818E-05,
0.000234954638898444,
0.000183683597200002,
5.17772978003706E-05,
0.000152398743200166,
0.0057978538092005,
0.00238549309760039,
3.81442874992644E-05,
4.7191598699925E-05,
6.2925821200821E-05,
0.00491710129900014,
0.0172100462598994,
0.00120445038490047,
0.3358983881224,
0.1235392857675,
0.00565574769139943,
8.67603166003761E-05,
0.00835230445179924,
0.000515782749801019,
0.405296137505099,
0.0318634475803012,
2.58446024989922E-05,
0.068493203960001,
0.0241485266921,
0.0152045342716995,
7.52199219995475E-05,
0.0102685189766998,
8.54099580038792E-06,
0.263607882214501,
0.0786934293236996,
0.000235844007999475,
0.173462964675,
0.00428034395249988,
2.51317591999367E-05
}

\def\xallDNA{
2.4081780613503,
2.9322648887841,
4.0213486915139,
4.1649595415673,
4.3805650311779,
4.6906262235970,
5.0468779807489,
5.6585881720011,
5.8913205790136,
6.3110106343988,
6.4730117758303,
6.5726197365799,
6.7855378677526,
7.1092082343074,
7.1566233336214,
7.1712025063763,
7.4278255755532,
7.5179840454278,
7.6016502045499,
7.9364126707962,
7.9684082064299,
8.0351268676340,
8.1456339213038,
8.3372239575289,
8.4610236405675,
8.4729078569858,
8.6203631402943,
8.6395223567454,
8.7083601831044,
8.9175176198805,
8.9730015493976,
9.0197060136468,
9.1952659652165,
9.4278902227017,
9.5246111842488,
9.6293159877464,
9.8184622567801,
9.8245280525607,
10.0620932130896,
10.1409637237163,
10.2368904942092,
10.3358471045367,
10.4722855948635,
10.4882393428124,
10.5566393881305,
10.7006448708082,
10.7241239405657,
10.7313064203900,
10.8368298559381,
10.9270960498152,
10.9381734411772,
10.9863706920757,
11.0594121896795,
11.1322403515608,
11.1659981648676,
11.1954356911576,
11.2470341910076,
11.3848721822811,
11.4267022851815,
11.4623049026606,
11.5075541731719,
11.5498051825465,
11.6216954979924,
11.7869458084520,
11.8282831354586,
11.9168247043002,
11.9384352754561,
12.0153110243669,
12.0659685612589,
12.1139829770217,
12.1606076852061,
12.2170808505215,
12.2970214604053,
12.3138888133021,
12.3436522624368,
12.3612666951092,
12.4849930409032,
12.4995214662693,
12.6602032143763,
12.7075203403849,
12.8927235040974,
12.9152231332033,
12.9526435830857,
12.9628523179737,
13.1178777634134,
13.1631798452539,
13.1743226857619,
13.2026790706030,
13.2359027662864,
13.2410678950881,
13.4323339071593,
13.5138663054393,
13.5546039445207,
13.5781147206539,
13.6523373517591,
13.6967883021372,
13.7136255876771,
13.7286925208841,
13.8454125007302,
13.8568510813213
}

\def\xincDNA{
0,
0.524086827433827,
1.089083802729795,
0.143610850053383,
0.215605489610604,
0.310061192419134,
0.356251757151827,
0.611710191252251,
0.232732407012487,
0.419690055385158,
0.162001141431487,
0.099607960749658,
0.212918131172693,
0.323670366554790,
0.047415099314033,
0.014579172754867,
0.256623069176928,
0.090158469874555,
0.083666159122142,
0.334762466246241,
0.031995535633741,
0.066718661204077,
0.110507053669794,
0.191590036225126,
0.123799683038643,
0.011884216418220,
0.147455283308581,
0.019159216451056,
0.068837826358989,
0.209157436776142,
0.055483929517111,
0.046704464249176,
0.175559951569662,
0.232624257485252,
0.096720961547087,
0.104704803497585,
0.189146269033733,
0.006065795780561,
0.237565160528860,
0.078870510626702,
0.095926770492925,
0.098956610327493,
0.136438490326821,
0.015953747948950,
0.068400045318041,
0.144005482677743,
0.023479069757492,
0.007182479824298,
0.105523435548035,
0.090266193877170,
0.011077391361964,
0.048197250898468,
0.073041497603867,
0.072828161881317,
0.033757813306737,
0.029437526290044,
0.051598499849961,
0.137837991273503,
0.041830102900443,
0.035602617479046,
0.045249270511366,
0.042251009374580,
0.071890315445870,
0.165250310459585,
0.041337327006637,
0.088541568841578,
0.021610571155922,
0.076875748910751,
0.050657536892066,
0.048014415762772,
0.046624708184396,
0.056473165315394,
0.079940609883820,
0.016867352896828,
0.029763449134650,
0.017614432672387,
0.123726345794058,
0.014528425366040,
0.160681748107047,
0.047317126008625,
0.185203163712419,
0.022499629105932,
0.037420449882367,
0.010208734888037,
0.155025445439675,
0.045302081840509,
0.011142840508041,
0.028356384841112,
0.033223695683363,
0.005165128801666,
0.191266012071223,
0.081532398280045,
0.040737639081387,
0.023510776133179,
0.074222631105192,
0.044450950378112,
0.016837285539866,
0.015066933207068,
0.116719979846035,
0.011438580591108
}

\title{A structure-preserving Chebyshev-filtered subspace iteration for the Bethe--Salpeter eigenvalue problem\thanks{This work was supported by grant PID2022-139568NB-I00 funded by MCIN/AEI/10.13039/501100011033 and by ERDF/EU. Innovation Study ISOLV-BSE has received funding through the Inno4scale project, which is funded by the European High-Performance Computing Joint Undertaking (JU) under Grant Agreement No 101118139. The JU receives support from the European Union's Horizon Europe Programme. The first author was also supported by Universitat Politècnica de València in its PAID-01-23 program.}
}

\author{Blanca Mellado-Pinto\thanks{D.\ Sistemes Inform\`atics i Computaci\'o, Universitat Polit\`ecnica de Val\`encia, Val\`encia, Spain
  (\texttt{bmelpin@dsic.upv.es}).}
\and Fernando Alvarruiz\thanks{D.\ Sistemes Inform\`atics i Computaci\'o, Universitat Polit\`ecnica de Val\`encia, Val\`encia, Spain
  (\texttt{fbermejo@dsic.upv.es}).}
\and Jose E. Roman\thanks{D.\ Sistemes Inform\`atics i Computaci\'o, Universitat Polit\`ecnica de Val\`encia, Val\`encia, Spain
  (\texttt{jroman@dsic.upv.es}).}
}

\begin{document}

\maketitle

\begin{abstract}
The Bethe--Salpeter equation, which has many applications in both theoretical and applied physics, is generally solved via a matrix eigenvalue problem with a rich algebraic structure. The numerical solution of such structured eigenproblem calls for specific algorithms that are able to preserve the structure throughout the computation. Several structure-preserving methods have already been proposed in the literature. In this paper, we develop a polynomial filter strategy that is able to extract approximations of eigenvalues located inside a specified interval. For this, we have devised a structure-preserving Chebyshev polynomial series, along with a specialized subspace iteration method that preserves the Bethe--Salpeter structure at every step of the algorithm. All necessary details required for a robust implementation are incorporated, and the performance is illustrated with matrices arising from real applications.
\end{abstract}

\section{Introduction}\label{sec:intro}

Optical absorption spectroscopy has long been used to obtain information about the structure and properties of elements, molecules and materials.
Nowadays physical experiments are either replaced or combined with computational experiments which obtain the optical absorption spectra \textit{ab initio} (i.e., from the known fundamental properties of matter).
The current tools available for this purpose are codes such as Yambo~\cite{Marini:2009:YAI}, BerkeleyGW~\cite{Deslippe:2012:BMP}, and others. The information that can be obtained from these simulations is valuable, for example, for improving photoelectric technology such as solar cells or optoelectronic devices.

In optical absorption, a sample is perturbed by photons, exciting the electrons into a conduction state, still in the system and interacting with the holes they leave behind. This particularity has driven the need to go beyond established frameworks like Density Functional Theory (DFT), which work well for the computation of ground states, and devise new theories such as Many Body Perturbation Theory (MBPT), that introduces the concept of pseudoparticles such as the electron-hole, also called exciton~\cite{Onida:2002:EED}. Without accounting for this, \textit{ab initio} results fail to accurately match experimental observations in some cases, notably semi-conductors and insulators.

The bound state of the exciton is computed by solving the Bethe--Salpeter equation (BSE), which involves solving an eigenvalue problem with a Hamiltonian $H$. The matrix $H$, which is typically dense, can become quite large and require significant computational effort.
For example, Gosetti \textit{et al}.~\cite{Gosetti:2024:DCE} employ a matrix of dimension 40960.
A few eigenvalues can be enough to approximate the absorption spectra, and if the number of requested eigenvalues is low, the full diagonalization of the matrix can be avoided, and an iterative solver is preferred. However, the total cost incurred by the operations with a dense matrix is also a factor that needs to be considered when selecting the method to be used.

The Bethe--Salpeter eigenvalue problem is interesting from a numerical linear algebra point of view because of the structure of the matrix, which is J-symmetric as well as pseudo-Hermitian. Although the problem can be treated simply as non-Hermitian, it is possible to take advantage of such structure to develop structure-preserving solvers with better properties regarding computational efficiency and memory requirements.

While some attention has already been given to creating such methods, in this paper we focus on the idea of polynomial filtering and how it affects the structure of the matrix.
There are a few reasons to use a polynomial filter. The first one is to be able to select a specific region of interest, possibly in the interior of the spectrum. The second is to accelerate the convergence of methods based on iteration of a subspace. A third reason is to develop spectrum slicing methods that compute different parts of the spectrum independently, having great potential for more scalable parallel computing.

In particular, we have devised a structure-preserving filter based on a Chebyshev polynomial series. The objective is to keep the matrix structure if the filter were applied explicitly to a BSE matrix. This is fundamental to introduce this technique in structure-preserving methods.

Even though the main goal of this article is to showcase the filter, which can be integrated into existing methods, we have devised a structure-preserving subspace iteration method as well. Subspace iteration methods in practice require polynomial acceleration to be effective.

This paper is structured in the following way. \Cref{sec:prelim} will introduce the problem and properties of the matrix, as well as the state of the art on methods for this problem. It will also briefly review the subspace iteration method and some polynomial filtering concepts. \Cref{sec:bsefilt} will present the proposed structure-preserving filter and its theoretical foundations. \Cref{sec:subspace-iteration} will describe an also novel structure-preserving subspace iteration method to showcase the potential of the filter. \Cref{sec:clenshaw} will describe how to apply the filter implicitly in the context of two different methods using modified versions of the Clenshaw algorithm. \Cref{sec:results} will discuss some computational results. Lastly, \Cref{sec:concl} will offer some concluding remarks.

{\textit{Notation:} For a block structured matrix, we will use $X_i$ for the block $i=1,2,\dots$ of matrix $X$. 
For a matrix $X$, $x_i$ denotes the column $i$ of the matrix, while for a vector $y$, $y_i$ denotes the element with index $i$. The superscript $X^{(i)}$ or $x^{(i)}$ is used for indexing a matrix or a vector in a sequence of similar elements obtained from previous ones, for example, the elements of a recurrence.

The superscript $\cdot^*$ denotes conjugate transposition, and the superscript $\cdot^T$ transposition.
We will use $\succ 0$ to express the positive definiteness of a matrix. 

We will denote by $\text{span}(X)$ the subspace generated by the columns of $X$, 
by $\text{diag}(X)$ a matrix whose diagonal entries are the diagonal entries of $X$ and the rest are zeros,
by $\text{diag}(d_1,\dots,d_k)$ a diagonal matrix whose diagonal entries are listed explicitly,
by $\text{sign}(X)$ the matrix sign function, and
by $|X|$ the element-wise absolute value of a matrix.

\section{Preliminaries}\label{sec:prelim}

\subsection{The Bethe--Salpeter eigenproblem}\label{sec:bse}
The goal is to compute a subset of eigenvalues $\lambda_i$ and eigenvectors $x_i$ of the eigenproblem
\begin{equation}\label{eq:eigenproblem}
Hx_i=\lambda_i x_i,
\end{equation}
where $ H \in \mathbb{C}^{2m\times 2m} $ and has the block structure
\begin{equation} \label{eq:bse-structure}
H =
 \begin{bmatrix}
      R & C \\
   -C^* & -R^T
 \end{bmatrix},
\quad
R = R^*,
\quad
C = C^T.
\end{equation}
From the application point of view, the block $R$ is the resonant term and $C$ the coupling term.
From now on and for convenience, $H$ will also be referred to as the Bethe--Salpeter equation matrix (or BSE matrix) and its block structure~\eqref{eq:bse-structure} as the BSE structure.

One can observe that $H$ is not Hermitian, hence if the eigenproblem is solved without taking into account any additional properties, a generic non-Hermitian method should be used.
As the complexity of the studied system grows, so do the computational needs, making it valuable to find alternative ways to solve the problem in a cost-efficient manner.
A common approach is to use the Tamm--Dancoff approximation, which discards the coupling term, thus reducing the problem to computing the eigenpairs of the Hermitian block $R$. However, this approximation is not always suitable, as shown, e.g., by Gr{\"u}ning \textit{et al}.~\cite{Gruning:2009:EPS}, with instances of nanomaterials where the approach fails, showcasing the need to exploit additional properties of the matrix structure to solve the full problem in a more effective way.

In particular, $H$ belongs to a class of matrices $H_\text{C}$, called J-symmetric, that satisfy $JH_\text{C}={(JH_\text{C})}^T$, where
$J=
  \left[\begin{smallmatrix}
      & I \\
    -I & 
  \end{smallmatrix}\right]
$.
Its properties are well known, for example they are described in~\cite{Bunse-Gerstner:1992:CNM,Mackey:2003:STS,Benner:2018:SRC}. The eigenvalues of a J-symmetric matrix appear in pairs $(\lambda_i,-\lambda_i)$.

$H$ also belongs to the class of pseudo-Hermitian matrices $H_{\text{PH}}$, that satisfy $\SignGen H_{\text{PH}} = {(\SignGen H_{\text{PH}})}^*$, where $\SignGen=\text{diag}\{\pm 1\}$. A subclass of pseudo-Hermitian matrices are those where $\SignGen$ is 
$\SignGen_{p,q} =
\left[\begin{smallmatrix}
I_p & \\
    & -I_q
\end{smallmatrix}\right],
$
a signature matrix with $p$ positive entries followed by $q$ negative ones.
In the case of the BSE matrix, $p=q=m$. We will denote by
$\SignBSE=
  \left[\begin{smallmatrix}
    I_m &   \\
     & -I_m
  \end{smallmatrix}\right]
$
this signature matrix, and by extension denote by $\Sigma$ and a size subindex any signature matrix with the same property but different size.
The eigenvalues of pseudo-Hermitian matrices are either real or appear in complex conjugate pairs $(\lambda_i,\overline{\lambda_i})$. 

Since a BSE matrix satisfies both properties, its eigenvalues appear as purely real or purely imaginary pairs $(\lambda_i,-\lambda_i)$ or in quadruples $(\lambda_i,-\lambda_i,\overline{\lambda_i},-\overline{\lambda_i})$. Additionally, the eigenvectors also exhibit structural properties. In this article we will not examine the properties of the eigenvectors of the general BSE matrix, focusing instead on the most common case in applications, where an additional restriction applies, as explained below.

The eigenproblem~\eqref{eq:eigenproblem} can be converted into a generalized eigenvalue problem $\SignBSE x_i=\lambda_i^{-1}\hat{H}x_i$, where
\[
 H = \SignBSE\hat{H},\quad
 \hat{H}=\begin{bmatrix}
    R & C \\
  C^* & R^T
 \end{bmatrix},
\]
and both $\SignBSE$ and $\hat{H}$ are Hermitian.
In addition, if $\hat{H} \succ 0$ all the eigenvalues are real and $H$ can be diagonalized as~\cite[Theorem 3]{Shao:2016:SPP}
\begin{equation}\label{eq:bse-diagonalization}
\begin{gathered}
HX = X\Lambda, \qquad Y^*H = \Lambda Y^*, \qquad Y^*X=I_{2m}, \\
\Lambda = 
\begin{bmatrix}
  \Lambda_+ &  \\
   & -\Lambda_+
\end{bmatrix},
\qquad
X = 
\begin{bmatrix}
  \XX & \oXY \\
  \XY & \oXX
\end{bmatrix},
\qquad
Y = 
\begin{bmatrix}
   \XX & -\oXY \\
  -\XY & \oXX
\end{bmatrix}.
\end{gathered}
\end{equation}
In this paper we always assume $\hat{H} \succ 0$. This is referred to in the literature as a definite BSE matrix $H$. This restriction forces the eigenvalues of $H$ to appear in real pairs $(\lambda_i,-\lambda_i)$.
Additionally, the right eigenvectors associated with $-\lambda_i$ and the left eigenvectors of $\pm\lambda_i$ can be obtained from the right eigenvectors associated with $\lambda_i$, according to \eqref{eq:bse-diagonalization}.

Several authors have devised structure-preserving methods for the definite BSE eigenproblem that take advantage of these properties.
An example of a technique that preserves the structure would be to project the matrix in a way that the projected matrix inherits some property, in this case the realness and the symmetry of the spectrum. This approach also leads to methods that operate only with some blocks of the data structures, handling the rest of the relations implicitly, as the structure is preserved throughout, resulting in more efficient methods from the computational point of view.

First, Gr{\"u}ning \textit{et al}.~\cite{Gruning:2011:ITL} expanded their aforementioned work, formalizing a Lanczos-based method that computes the optical spectra in a structure-preserving way, using the inner product induced by $\hat{H}$.
Shao \textit{et al}.~\cite{Shao:2016:SPP} proposed a method to diagonalize $H$ via a transformation to a real Hamiltonian, and later~\cite{Shao:2018:SPL} introduced another iterative Lanczos-based method, also showing its relationship with the Gr{\"u}ning algorithm and other methods.
In our previous work~\cite{Alvarruiz:2025:VTR}, we have extended the methods in~\cite{Shao:2018:SPL} to include a structure-preserving thick restart procedure.
Shan and Shao have recently proposed a structure-preserving LOBPCG method~\cite{Shan:2026:SPL}.

Let us take the Lanczos method described in~\cite{Shao:2018:SPL} as an example.
It requires the structure~\eqref{eq:bse-structure} of the matrix and also enforces the structure
$[u^T, \varsigma\bar{u}^T]^T$ on the basis vectors, with a sign $\varsigma=\pm 1$, which makes it possible to operate only with the upper half of the vectors, taking into account that
\begin{equation}\label{eq:H-times-u-conj-u}
H\begin{bmatrix}u \\ \varsigma\overline{u}\end{bmatrix}
= \begin{bmatrix}
  R u + \varsigma C\overline{u} \\
  -\mybar{C}u - \varsigma \mybar{R}\overline{u}
  \end{bmatrix}
= \begin{bmatrix}
  R u + \varsigma C\overline{u} \\
  -\varsigma(\varsigma\mybar{C}u + \mybar{R}\overline{u})
  \end{bmatrix}
=
\begin{bmatrix}v \\ -\varsigma\overline{v}\end{bmatrix}.
\end{equation}
Furthermore, this method also assumes positive definiteness of $\hat{H}$.

In this context, any spectral transformation to be applied should take care of preserving the same structure and properties required by the method.
The inverse of a definite BSE matrix is also a definite BSE matrix~\cite{Shao:2017:PDB}. We can also prove that the inverse of a general BSE matrix has BSE structure (\cref{sec:inverse-H}).
However, a shift-and-inverse transformation with a non-zero shift is not structure-preserving, since the matrix $(H-\sigma I)^{-1}$ loses the relationship between the top left and bottom right blocks.

The motivation for this work is to extend this idea of a structure-preserving spectral transformation to polynomial filtering, which pursues the same objective, selecting a filter that preserves the structure and properties of the matrix, and could be used in combination with any structure-preserving eigensolver.

Although we use this example (the Lanczos method described in~\cite{Shao:2018:SPL}) to illustrate the challenge of integrating a filter into structure-preserving methods, using polynomial filters for accelerating the convergence of this method has not shown any advantage, especially when dense matrices are involved, as is generally the case in applications. Conversely, subspace iteration methods like the structure-preserving method later described in this paper benefit from using polynomial filtering. That is why we have chosen to explore the idea of structure-preserving Chebyshev filtered subspace iteration, even when the filter is general and could be applied to other methods. In fact, recently Di Napoli \textit{et al.}~\cite{Napoli:2026:CAS} have also proposed a filtered subspace iteration for the BSE, using a similar oblique projection, but employing a low-pass filter applied to the folded spectrum, by applying it to $H^2$. A low-pass filter requires a lower degree polynomial and is suited for computing eigenvalues on one end of the spectrum, while the method that we propose with a structure-preserving band-pass filter requires a higher degree polynomial but can obtain the eigenvalues in any subinterval of the spectrum, which is also a first step for devising a spectrum slicing method.

\subsection{Subspace iteration}\label{sec:si}
The subspace iteration method~\cite{Bai:2000:TSA} is based on a vector simultaneous iteration combined with a Rayleigh-Ritz procedure. For numerical stability, the vectors must be mutually orthonormalized, although it is not necessary to do so every iteration, which is a common optimization.

\cref{alg:subspace-iteration} shows a simple version of the method for a matrix $A \in \mathbb{C}^{m\times m}$. If $A$ is Hermitian, normal, or can be diagonalized with a well-conditioned similarity transformation, then the eigenvectors will be computed in a numerically stable way. For the more general non-Hermitian case, a few modifications in the algorithm are required, notably step~\ref{algl:si-diagonalization} should be turned into a Schur decomposition instead of a diagonalization. Note that $A$ can be replaced by $A^{-1}$ (inverse iteration) to make the eigenvalues closest to zero dominant. Alternatively, one could apply a filter, for example a polynomial one, to accelerate the convergence of specific eigenvalues, as discussed in~\cref{sec:filt}. This can be done by premultiplying $\AX$ with a polynomial $p(A)$ before line \ref{algl:si-orthog}.

It is best to employ a subspace with a number of vectors $q$ larger than the number of desired eigenpairs $\ell$, since it has been proven that the rate of convergence for an eigenvalue $\lambda_i$, with $|\lambda_1|\ge|\lambda_2|\ge\dots\ge|\lambda_q|>|\lambda_{q+1}|\ge\dots\ge|\lambda_m|$, contained in the subspace, is governed by $\frac{|\lambda_{q+1}|}{|\lambda_{i}|}$ for $i\leq q$~\cite{Saad:2011:NML}. When choosing the optimal value for $q$ in a real implementation, a balance must be sought to improve the convergence without increasing too much the cost of the operations. A usual rule of thumb is to take $q=2\ell$.

One comment is that the convergence in~\cref{alg:subspace-iteration} is checked with the residual of $X$, since the product $AX$ for computing $\mathcal{R} = AX-X\Lambda$ is already available in $\AX$. That means that each iteration checks the convergence of the subspace computed in the previous iteration. The method we employ is to check that the relative residual, $||r_i||/|\lambda_i|$ is under a given tolerance.
In line~\ref{algl:si_convergence}, at least the first $\ell$ eigenpairs must have converged to exit.

Several acceleration techniques have been proposed for this method, like the Atkinson acceleration, or overrelaxation, but they are outside the scope of our work. The selection of the initial subspace also has an important impact on the convergence of the method, since an initial subspace close to the wanted invariant space will converge faster to it.

\begin{algorithm}[t]
  \caption{Subspace iteration}
  \label{alg:subspace-iteration}
  \begin{algorithmic}[1]
	  \REQUIRE Matrix $A \in \mathbb{C}^{m\times m}$, initial guess $X^{(0)} \in \mathbb{C}^{m\times q}$, number of wanted eigenvalues $\ell<q$
	  \ENSURE Approximate eigenpairs $(X,\Lambda)$, where $X \in \mathbb{C}^{m\times \ell}$ and $\Lambda=\text{diag}(\lambda_1,\lambda_2,\dots,\lambda_\ell) \in \mathbb{R}^{\ell\times \ell}$
    \STATE Set $\AX=X^{(0)}$
	  \WHILE{not converged}
			\STATE $Q = \text{Orthogonalize}(\AX)$ \label{algl:si-orthog}
			\STATE $\AQ = AQ$
			\STATE $M = Q^*\AQ$
			\STATE Compute eigendecomposition $MW = W\Lambda$ \label{algl:si-diagonalization}
			\STATE Sort $(W,\Lambda)$ in descending order of magnitude of the diagonal elements of $\Lambda$
			\STATE $\AX = \AQ W$
			\STATE $X = QW$
			\STATE $\mathcal{R} = \AX-X\Lambda$ 
			\STATE Check convergence of the first $\ell$ Ritz pairs in $(X,\Lambda)$ \label{algl:si_convergence}
		\ENDWHILE
		\STATE Set $X=$ first $\ell$ columns of $X$, and $\Lambda=$ leading $\ell\times\ell$ block of $\Lambda$
  \end{algorithmic}
\end{algorithm}

Even though Krylov methods are known to have a faster convergence, this method is still widely used in libraries and physics codes, especially in combination with polynomial filter acceleration, for example~\cite{Winkelmann:2019:CHS}. Actually, there are advantages of this method over Lanczos iteration, which arise from the fact that the entire subspace is updated in each step. One advantage is that it can accommodate perturbations in the matrix \cite{Saad:2016:ASI}, since the last computed subspace can be used as an initial subspace for the perturbed matrix. There is also a computational advantage, related to the concept of arithmetic intensity, which is of especial importance for GPU computation. Subspace iteration employs matrix-matrix products, a BLAS-3 operation, instead of matrix-vector products, which are BLAS-2. This classification refers to the relationship between the order of the number of data elements and the number of operations performed with that amount of data. Since reading and writing data from/to memory is much slower than performing arithmetic operations, carrying out a larger number of operations with the available data is computationally preferable.

\subsection{Polynomial filters}\label{sec:filt}

Given a Hermitian matrix $A \in \mathbb{C}^{m\times m}$ whose spectrum is contained in the interval $[\lambda_\text{min},\lambda_\text{max}]$, polynomial filtering is a technique used to compute the eigenvalues located inside a subinterval $[\alpha,\beta] \subset [\lambda_{\text{min}},\lambda_\text{max}]$, avoiding the factorization of the matrix. The eigenvalues of $p(A)$ are computed instead, where $p$ is a polynomial that approximates the indicator function
\begin{equation}\label{eq:indicator-function}
\phi_{[\alpha,\beta]}(x)=\left\lbrace \begin{array}{ll}
1, & x\in \left\lbrack \alpha ,\beta \right\rbrack, \\
0, & x\not\in \left\lbrack \alpha ,\beta \right\rbrack,
\end{array}\right.
\end{equation}
inside $[\lambda_{\min} ,\lambda_{\max}]$.
The objective is to map the eigenvalues inside the subinterval to be the dominant ones, and therefore the ones to converge first.  

In this work we consider only band-pass filters, intended to compute interior eigenvalues, but it is also possible to define low- or high-pass filters to accelerate the computation of exterior eigenvalues.
In this section, we provide a brief review of the essential theory, but this theoretical background can also be found in other articles on the same topic~\cite{Schofield:2012:SSM,Aurentz:2017:CGI,Fang:2012:FLP}.

Suppose we have a partial eigendecomposition $AX_\ell=X_\ell\Lambda_\ell$, then the polynomial function $p$ satisfies $p(A)X_\ell=X_\ell\text{diag}(p(\lambda_1),\dots,p(\lambda_\ell))$. The subspace generated by the columns  of $X_\ell$ is invariant under both $A$ and $p(A)$, so approximations of the eigenvalues and eigenvectors of $A$ can be extracted via a Rayleigh-Ritz procedure from the converged approximation of the subspace $X_\ell$ computed by subspace iteration with $p(A)$.

We will use a series of Chebyshev polynomials to build $p$. This imposes a restriction that the spectrum is contained in $[-1,1]$, as we shall see. Therefore, before proceeding, the matrix needs to be scaled for its spectrum to be contained inside this range. The scaling can be done as
\begin{equation}\label{eq:scaling}
\tilde{A} = \left(\frac{A-\sigma I}{\rho}\right),\quad
\sigma=\frac{\lambda_{\text{max}}+\lambda_{\text{min}}}{2},\quad
\rho=\frac{\lambda_{\text{max}}-\lambda_{\text{min}}}{2}.
\end{equation}
In finite precision arithmetic, and particularly if $\lambda_{\text{max}}$ is computed without too much accuracy, it may be safer to increase it by an amount $\delta>||r||_2$, where $r$ is the residual associated with $\lambda_{\text{max}}$ (and similarly for $\lambda_{\text{min}}$), to ensure that the spectrum is truly contained in $[-1,1]$ after scaling.
In the rest of this section, we use $A$ to refer to the scaled matrix $\tilde A$, with $[\alpha, \beta]$ the corresponding scaled subinterval.

Recall that Chebyshev polynomials of the first kind $T_n(\cdot)$ can be computed with a three-term recurrence relation,
\begin{gather}\label{eq:chebyshev-terms}
T_{n+1}(A) = 2AT_n(A) - T_{n-1}(A), \nonumber \\
T_0(A) = I, \\
T_1(A) = A \nonumber.
\end{gather}

It is well known from approximation theory, that a function defined in the interval $[-1,1]$ can be approximated with a finite (truncated) sum of Chebyshev polynomials. In case of a matrix function, the spectrum of the matrix must be defined in $[-1,1]$ and the series is
\begin{equation} \label{eq:chebyshev-series}
p(A) = \sum_{n=0}^{N}{c_nT_n(A)},
\end{equation}
where $N$ is the degree of the polynomial and $c_n$ are the coefficients. If the function is smooth, the approximation requires only a few terms. However, when approximating the indicator function~\eqref{eq:indicator-function}, which is discontinuous, exponential convergence is lost and also the Gibbs phenomenon appears, as discussed below.

The coefficients for the Chebyshev series approximation of function $\phi$ can be computed as
\begin{equation}\label{eq:coeffs}
c_n = \left\lbrace \begin{array}{ll}
\frac{2(\sin(n\arccos(\alpha))-\sin(n\;\arccos(\beta)))}{\pi n}, & n>0, \\
\frac{\arccos(\alpha)-\arccos(\beta)}{\pi}, & n=0.
\end{array}\right.
\end{equation}
These coefficients are determined only by the values of the endpoints of the subinterval of interest, $\alpha$ and $\beta$. We will use the notation $p_{[\alpha,\beta]}(A)$ when it is necessary to emphasize which subinterval characterizes the polynomial.

\pgfplotsset{ every non boxed x axis/.append style={x axis line style=-},
     every non boxed y axis/.append style={y axis line style=-}}
\def\lambdamax{5.865}
\def\lowerbound{1}
\def\upperbound{3}
\begin{figure}
\label{fig:filter}
\centering
\begin{tikzpicture}[scale=1]
  \begin{axis}[
  axis lines=center,
  height = 4cm, width = 8cm, scale only axis,
  ymin = -3, ymax = 7,
  xtick=\empty,
  ytick=\empty,
  xtick={0},
  xticklabels={0},
  hide obscured x ticks=false,
  ticklabel style={font=\small,fill=white},
  ]
  \addplot [smooth,gray] table [x expr=-\lambdamax+0.01*\coordindex,y={y_none},col sep=comma] {filterplotdata.txt};
  \addplot [smooth] table [x expr=-\lambdamax+0.01*\coordindex,y={y_jackson},col sep=comma] {filterplotdata.txt};
  \node [yshift=-15] at (axis cs:\lowerbound,0)  {$\ub$};
  \node [yshift=-15] at (axis cs:\upperbound,0)  {$\lb$};
  \end{axis}
\end{tikzpicture}
\caption{Representation of p(x) where p is a Chebyshev series polynomial filter of degree 200 applied to domain $[-1,1]$. The lighter line corresponds to the filter with no damping applied. The peaks near $\alpha$ and $\beta$ are the Gibbs oscillations produced when approximating a discontinuous function. The darker line represents the same polynomial filter but employing the Jackson kernel.}
\end{figure}
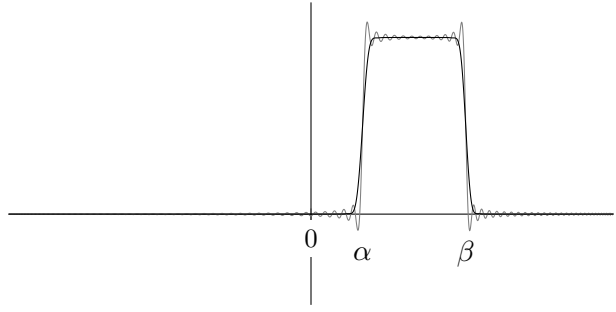

A representation of the filter can be observed in~\cref{fig:filter}.
This approximation is affected by Gibbs oscillations, a phenomenon resulting, in this case, from the finite approximation of a discontinuous function. The figure shows that oscillations appear near the discontinuities of $\phi$.
To eliminate these oscillations, it is enough to add some damping factors~\cite{Weisse:2006:KPM,Pieper:2016:HPI}, $g_n$, that modify the coefficients in the series, so that
\begin{equation} \label{eq:chebyshev-series-damping}
p(A) = \sum_{n=0}^{N}{g_nc_nT_n(A)}.
\end{equation}

Wei{\ss}e and coauthors~\cite{Weisse:2006:KPM} use the concept of kernel to study the truncation of the Chebyshev series to a finite series of order $N$ and the modification of the coefficients via the damping factors, which can be expressed together as a convolution of the original function to approximate and such kernel $K_N$. We are particularly interested in the study of the positivity property of the kernel, $K_N(x,y)>0, \quad\forall x,y \in [-1,1]$, which, if satisfied, ensures that approximations of positive quantities remain positive. For example, $\phi$ is never negative, therefore using a damping that satisfies a positive kernel implies that the approximate polynomial will never be negative. Intuitively, we can see that in~\cref{fig:filter} the smoothed polynomial function that approximates $\phi$ built using a Jackson kernel does not cross the x-axis to the negative space, while the Gibbs oscillations present in the polynomial approximation with no damping do.
\cref{tab:damping} presents different alternatives for the damping coefficients, with an indication of whether the kernel satisfies the positivity property.

\begin{table}%
\caption{Computation of the damping coefficients $g_n$ of different kernels and whether they satisfy the positivity property.}
\label{tab:damping}
\begin{center}
\begin{tabular}{lcc}
\hline
kernel & $g_n$ & positive \\
\hline
None & 1 & no \\
Fejer & $1-\frac{n}{N+1}$ & yes \\
Jackson & $\frac{1}{N+2}((N+2-n)\cos{\frac{\pi n}{N+2}}+\sin{\frac{\pi n}{N+2}}\cot{\frac{\pi}{N+2}})$ & yes \\
Lanczos & ${\left(\frac{\sin{(\pi n /(N+1))}}{\pi n/(N+1)}\right)}^{\gamma}$ & no \\
\hline
\end{tabular}
\end{center}
\end{table}%

The filter separates the eigenvalues inside the subinterval $[\alpha,\beta]$ from those outside by its magnitude, since eigenvalues inside the subinterval should be mapped to a value close to one and the rest to a value close to zero. In practice, eigenvalues that are close to the subinterval endpoints turn into a value ranging from zero to one. A threshold can be used to decide if an eigenvalue is inside the subinterval, computed as $p(\alpha)$ (or $p(\beta)$). The filter itself can also be scaled by dividing the coefficients by $2p(\alpha)$ \cite{Xi:2016:CPS} so that the threshold is set to 0.5. All eigenvalues of $p(A)$ above the threshold are assumed to be contained in the subinterval.

Several factors can influence the quality of the filter and therefore the convergence of the eigenvalues in the subinterval using a method like subspace iteration. A filter of a higher degree approximates the function better, and in particular a steep filter does a better job at separating the eigenvalues contained in the subinterval from the rest. However, increasing the degree is costly and some attemps have been made to reduce the degree of the filter while maintaining a good convergence rate. Pieper \textit{et al.}~\cite{Pieper:2016:HPI} provide an analysis on the optimal choice of the degree and number of search vectors, which also influences the efficiency of polynomial filtering methods. An elegant solution is to employ an adaptive filter, as in the work of Xu \textit{et al.}~\cite{Xu:2026:APF}, where the degree of the polynomial is expected to be reduced in successive iterations, governed by the estimated rate of convergence. Other proposed techniques rely on reducing the width of the area where unwanted eigenvalues are not sufficiently damped~\cite{Galgon:2017:ICP}.

When applying the filter, there are known methods to avoid the explicit evaluation of $p(A)$. In particular, when a function follows the three-term recurrence
\[T_{n+1}(A)=\omega(n,A)T_n(A)+{\varrho(n,A)T}_{n-1}(A)\]
the Clenshaw algorithm~\cite{Press:2002:NRC} can be applied as
\begin{gather*}
p(A)=\frac{1}{2}\omega(n,A)Y^{(1)} + \varrho(n,A)Y^{(2)} +c_0I,\\
Y^{(n)} =\omega(n,A)Y^{(n+1)} + \varrho(n,A)Y^{(n+2)} +c_nI,\\
Y^{(N+2)} =Y^{(N+1)} =0.
\end{gather*}

Chebyshev polynomials $T_n(A)$ follow this recurrence with $\omega(n,A)=2A$ and $\varrho(n,A)=-I$, leading to the formula
\begin{equation} \label{eq:clenshaw-not-scaled}
\begin{gathered}
p(A)=AY^{(1)} - Y^{(2)} +c_0I,\\
Y^{(n)} =2AY^{(n+1)} - Y^{(n+2)} +c_nI,\\
Y^{(N+2)} =Y^{(N+1)} =0.
\end{gathered}
\end{equation}

This kind of polynomial filtering has been implemented, for example, in the Cucheb code~\cite{Aurentz:2017:CGI} and in the EVSL library~\cite{Li:2019:ESL} for real symmetric eigenvalue problems. Additionally, in the EVSL library, the spectrum is divided into subintervals, and the eigenvalues of each subinterval are extracted independently using polynomial or rational filters. This is a technique called spectrum slicing, whose goal is to obtain a large number of interior eigenpairs (thousands or tens of thousands) from a sparse matrix using this divide-and-conquer strategy.

\section{Polynomial filters for the BSE}\label{sec:bsefilt}

We will apply a polynomial filter to the BSE matrix $H$. Although it is not Hermitian, it is possible to use polynomial filtering due to the fact that $H$ is diagonalizable and its spectrum is real.

We aim at devising a polynomial filter that is structure-preserving, that is, where the matrix polynomial retains the algebraic properties of the argument matrix. For the specific case of BSE, we formalize it in the following definition.
\begin{definition}\label{def:sp-filter}
A structure-preserving filter for a definite BSE matrix $H$ is a function $q(H)$, where $H$ is the shifted-and-scaled matrix as in~\eqref{eq:scaling}, that preserves
\begin{itemize}
  \item[i.] the block structure of $H$ described in~\eqref{eq:bse-structure}, which induces the symmetry of the spectrum, and
  \item[ii.] the definiteness property $\SignBSE q(H)\succ0$, which is the analog of $\SignBSE H \succ 0$, and induces the realness of the spectrum.
\end{itemize}
\end{definition}

Both properties are required by structure-preserving methods such as those described in~\cite{Shao:2018:SPL} and~\cref{sec:subspace-iteration} below.

We propose the filter
\begin{equation}
  \label{eq:bse-filter}
  q_{[\alpha,\beta]}(H) = p_{[\alpha ,\beta]}(H)-p_{[-\beta ,-\alpha]}(H),
\end{equation}
where $p$ is the Chebyshev series described in~\eqref{eq:chebyshev-series-damping}.

This filter approximates the function $\phi_{[\alpha,\beta]}(H)-\phi_{[-\beta,-\alpha]}(H)$ making the corresponding negative eigenvalues dominant as well as the positive ones.
Another way to view this is that the filter approximates a new indicator function
\[\psi_{[\alpha,\beta]}(x)=\left\lbrace \begin{array}{ll}
-1, & x\in \left\lbrack -\beta ,-\alpha \right\rbrack, \\
1, & x\in \left\lbrack \alpha ,\beta \right\rbrack, \\
0, & x\not\in \left\lbrack -\beta ,-\alpha \right\rbrack \cup \left\lbrack \alpha ,\beta \right\rbrack.
\end{array}\right.\]
In this way, the symmetry of the spectrum is kept, since the structure-preserving methods rely on this property.

The corresponding series is
\begin{equation}\label{eq:bse-series}
q_{[\alpha,\beta]}(H)=
\sum_{n=0}^N{g_nc_n^+T_n(H)} - \sum_{n=0}^N{g_nc_n^-T_n(H)}
=\sum_{n=0}^N{g_n c_n T_n(H)},
\end{equation}
where $c_n:=c_n^+-c_n^-$. We use $c_n^+$ to denote the coefficients corresponding to $p_{[\alpha ,\beta]}(H)$, which filters the eigenvalues on the positive side of the spectrum, and $c_n^-$ for the coefficients of $p_{[-\beta ,-\alpha]}(H)$, which does the same to the corresponding negative eigenvalues.

This filter can be visualized in~\cref{fig:bsefilter}.

\def\lambdamax{5.865}
\def\lowerbound{1}
\def\upperbound{3}
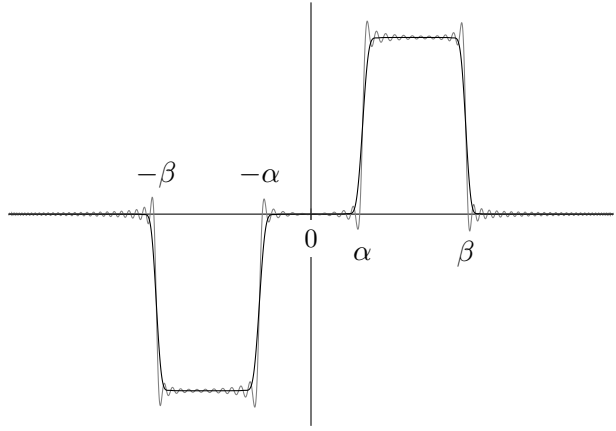
\begin{figure}
\label{fig:bsefilter}
\centering
\begin{tikzpicture}[scale=1]
  \begin{axis}[
  axis lines=center,
  height = 5.6cm, width = 8cm, scale only axis,
  ymin = -7, ymax = 7,
  xtick=\empty,
  ytick=\empty,
  xtick={0},
  xticklabels={0},
  hide obscured x ticks=false,
  ticklabel style={font=\small,fill=white},
  ]
  \addplot [smooth,gray] table [x expr=-\lambdamax+0.01*\coordindex,y={y_none},col sep=comma] {bsefilterplotdata.txt};
  \addplot [smooth] table [x expr=-\lambdamax+0.01*\coordindex,y={y_jackson},col sep=comma] {bsefilterplotdata.txt};
  \node [yshift=15]  at (axis cs:-\upperbound,0) {$-\lb$};
  \node [yshift=15]  at (axis cs:-\lowerbound,0) {$-\ub$};
  \node [yshift=-15] at (axis cs:\lowerbound,0)  {$\ub$};
  \node [yshift=-15] at (axis cs:\upperbound,0)  {$\lb$};
  \end{axis}
\end{tikzpicture}
\caption{Representation of $q(x)$, where $q$ is the structure-preserving filter, and the parameters are the same as in~\cref{fig:filter}. Again, the lighter line corresponds to no damping, and the darker line to the Jackson damping.}
\end{figure}

\subsection{Theoretical foundations of the filter}

To prove that the filter satisfies~\cref{def:sp-filter} we need some intermediate results.

\begin{lemma}\label{lma:coeffs}
Let $c_n=c_n^+ - c_n^- \in \mathbb{R}$, with $n \in \mathbb{Z}_{\geq 0}$, be a coefficient of the truncated Chebyshev polynomial series~\eqref{eq:bse-series} that approximates $\psi$ in $[-1,1]$. Then $c_n = 2c_n^+$ if $n= 2k+1 , k \in \mathbb{Z}_{\geq 0}$, and $c_n = 0$ if $n= 2k, k \in \mathbb{Z}_{\geq 0}$.
\end{lemma}
\begin{proof}
The coefficients are computed using the formula \eqref{eq:coeffs}. First, we prove the lemma for $n=0$,
\begin{align*}
c_0^- 
&= 
\frac{\arccos(-\beta )-\arccos(-\alpha)}{\pi } =
\frac{\pi -\arccos(\beta)-(\pi -\arccos(\alpha))}{\pi } \\ 
&=
\frac{\arccos(\alpha)-\arccos(\beta)}{\pi}.
\end{align*}

Since
\[c_0^+ =\frac{\arccos(\alpha)-\arccos(\beta)}{\pi},\]
it verifies that
\[c_0 = c_0^+ -c_0^- = 0.\]

Next, for the general case, we know that
\begin{align*}
c_n^-
&=
\frac{2(\sin(n\arccos(-\beta))-\sin (n\arccos(-\alpha)))}{\pi n} \\
&=
\frac{2(\sin(n\pi - n\arccos(\beta))-\sin(n\pi - n\arccos(\alpha)))}{\pi n}.
\end{align*}

Since
$\sin(n\pi-n\arccos(x)) = \sin(n\pi)\cos(n\arccos(x)) - \cos(n\pi)\sin(n\arccos(x))$
and taking into account that
$\sin(n\pi) = 0$ and $\cos(n\pi) = \begin{cases}-1, & n=2k+1,\\ 1, & n=2k,\end{cases}$

\[
c_n^- =
\begin{cases}
-\frac{2(\sin(n\arccos(\alpha))-\sin(n\arccos(\beta)))}{\pi n}, & n=2k+1, \\
\frac{2(\sin(n\arccos(\alpha))-\sin(n\arccos(\beta)))}{\pi n}, & n=2k.
\end{cases}
\]
Therefore, knowing
\[c_n^+ = \frac{2(\sin(n\arccos(\alpha))-\sin(n\arccos(\beta)))}{\pi n},\]
we can conclude that
\[c_n=c_n^+ - c_n^- =
\begin{cases}
2c_n^+, & n=2k+1, \\
0, & n=2k.
\end{cases}
\]
\end{proof}

\begin{lemma}\label{lma:odd-powers}
If $H \in \mathbb{C}^{2m\times 2m}$ is a matrix with structure as defined in~\eqref{eq:bse-structure}, then $H^{2k+1}, k \in \mathbb{Z}_{\geq 0}$, retains the same structure.
\end{lemma}
\begin{proof}
Let $H_A= \left[\begin{smallmatrix}
      R_A & C_A \\
   -C_A^* & -R_A^T
 \end{smallmatrix}\right]$ and $H_B= \left[\begin{smallmatrix}
       R_B & C_B \\
    -C_B^* & -R_B^T
  \end{smallmatrix}\right]$ be two matrices with structure~\eqref{eq:bse-structure} such that they commute. Then
\begin{align*}
H_C=H_AH_B &=
\begin{bmatrix}
R_AR_B-C_AC_B^*  & R_AC_B-C_AR_B^T \\
-C_A^*R_B+R_A^TC_B^* & -C_A^*C_B+R_A^TR_B^T
\end{bmatrix}
, \\
H_C=H_BH_A &=
\begin{bmatrix}
R_BR_A-C_BC_A^*  & R_BC_A-C_BR_A^T \\
-C_B^*R_A+R_B^TC_A^* & -C_B^*C_A+R_B^TR_A^T
\end{bmatrix}
.
\end{align*}

Defining the matrix blocks $
H_C=\begin{bmatrix}
B_1 & B_2 \\
B_3 & B_4
\end{bmatrix}
$, we have
\[
\begin{aligned}
B_1^T&={(R_AR_B - C_AC_B^*)}^T = -C_B^*C_A + R_B^TR_A^T=B_4,\\
-B_2^*&={-(R_AC_B -C_AR_B^T)}^* = -C_B^*R_A + R_B^TC_A^*=B_3,\\
B_1^*&={(R_AR_B - C_AC_B^*)}^* = R_BR_A - C_BC_A^*=B_1,\\
-B_2^T&={-(R_AC_B -C_AR_B^T)}^T = R_BC_A - C_BR_A^T = B_2,
\end{aligned}
\]
or, renaming $B_1$ and $B_2$ as $R_C$ and $C_C$, respectively,
\begin{equation}\label{eq:semibse-structure}
H_C=H_AH_B=
\begin{bmatrix}
R_C  & C_C \\
-C_C^*  & R_C^T
\end{bmatrix},
\quad R_C = R_C^*,
\quad C_C = -C_C^T.
\end{equation}

Note that $H_C$ does not satisfy the definition of BSE matrix due to the sign of the $(2,2)$-block and the skew-symmetry of $C_C$.

We now consider a matrix $H_A$ with structure~\eqref{eq:bse-structure} and a matrix $H_C$ with structure~\eqref{eq:semibse-structure} such that they commute,
\begin{align*}
  H_D=H_AH_C &=
  \begin{bmatrix}
    R_AR_C -C_AC_C^*  & R_AC_C +C_AR_C^T \\
    -C_A^*R_C + R_A^TC_C^*  & -C_A^*C_C - R_A^TR_C^T
  \end{bmatrix}, \\ 
  H_D=H_CH_A &=
  \begin{bmatrix}
    R_CR_A - C_CC_A^*  & R_CC_A - C_CR_A^T \\
    -C_C^*R_A - R_C^TC_A^*  & -C_C^*C_A - R_C^TR_A^T
  \end{bmatrix},
\end{align*}
and define the matrix blocks $
H_D=\begin{bmatrix}
\hat{B}_1 & \hat{B}_2 \\
\hat{B}_3 & \hat{B}_4
\end{bmatrix}
$. Then
\[
\begin{aligned}
-\hat{B}_1^T&=-{(R_AR_C - C_AC_C^*)}^T =-C_C^*C_A - R_C^TR_A^T=\hat{B}_4,\\
-\hat{B}_2^*&={-(R_AC_C + C_AR_C^T)}^* =-C_C^*R_A - R_C^TC_A^*=\hat{B}_3,\\
\hat{B}_1^*&={(R_AR_C - C_AC_C^*)}^* = R_CR_A - C_CC_A^*=\hat{B}_1,\\
\hat{B}_2^T&={(R_AC_C + C_AR_C^T)}^T = R_CC_A - C_CR_A^T=\hat{B}_2^T,
\end{aligned}
\]
which shows that $H_D$ has structure~\eqref{eq:bse-structure},
\begin{equation}\label{eq:product-bse-semibse}
H_D = H_AH_C =
\begin{bmatrix}
R_D  & C_D \\
-C_D^*  & -R_D^T
\end{bmatrix},
\quad R_D = R_D^*,
\quad C_D = C_D^T.
\end{equation}

If $H^{2k-1}$ has BSE structure, eq.~\eqref{eq:semibse-structure} implies that $H^{2k}=H^{2k-1}H=HH^{2k-1}$ has structure~\eqref{eq:semibse-structure}, and similarly, eq.~\eqref{eq:product-bse-semibse} shows that $H^{2k+1}=H^{2k}H=HH^{2k}$ has BSE structure. We can then prove by induction, starting from the base case $k=1 \mid H^{2k-1} = H$, with known BSE structure, that the lemma holds.
\end{proof}

Note that, given a definite BSE matrix $H$, scaling the matrix using~\eqref{eq:scaling} preserves the BSE structure, since taking into account that $\lambda_{\min}={-\lambda}_{\max}$, we have

\begin{equation}\label{eq:bse-scaling}
\begin{gathered}
  \sigma=\frac{\lambda_{\max} +\lambda_{\min}}{2}=0, \quad
  \rho=\frac{\lambda_{\max} -\lambda_{\min}}{2}=\lambda_{\max}, \\
  \tilde H=\frac{H-\sigma I}{\rho} = \frac{H}{\lambda_{\max}}.
\end{gathered}
\end{equation}
In the rest of this section, we use $H$ to refer to the scaled matrix $\tilde H$, with $[\alpha,\beta]$ the corresponding scaled subinterval.

\begin{theorem}\label{thm:p-is-bse}
Given a BSE matrix $H \in \mathbb{C}^{m\times m}$ with its spectrum inside $[-1,1]$, and a subinterval $[\alpha,\beta] \subset [0,1]$,
the matrix $q(H)=q_{[\alpha,\beta]}(H)$ has BSE structure~\eqref{eq:bse-structure}.
\end{theorem}
\begin{proof}
According to~\eqref{eq:bse-series},
\[
q_{[\alpha,\beta]}(H) = \sum_{n=0}^{N}{g_n(c_n^+-c_n^-)T_n(H)}.
\]
Using~\cref{lma:coeffs}, and assuming $N$ is odd, we can replace this series with
\[
q_{[\alpha,\beta]}(H) = \sum_{k=0}^{(N-1)/2}{2g_{2k+1}c_{2k+1}^+ T_{2k+1}(H)}.
\]

Chebyshev polynomials of odd degree $T_{2k+1}(H)$ can be expressed as a sum containing only scaled odd powers of $H$.
According to~\cref{lma:odd-powers}, odd powers of $H$ preserve structure~\eqref{eq:bse-structure}. A sum of scaled matrices preserves this structure too. Therefore $q_{[\alpha,\beta]}(H)$ preserves the structure of $H$.
\end{proof}

After proving that the filter preserves the BSE structure, we are concerned with the positive definiteness of $\SignBSE q(H)$.

\begin{theorem}
Let $H \in \mathbb{C}^{m\times m}$ be a BSE matrix with its spectrum inside $[-1,1]$ such that $\SignBSE H \succ 0$ and let $[\alpha,\beta] \subset [0,1]$. The matrix $q(H)=q_{[\alpha,\beta]}(H)$ verifies $\SignBSE q(H) \succ 0$, as long as $g_n$ in the series~\eqref{eq:bse-series} of $q$ are the coefficients of a kernel with the positivity property.
\end{theorem}
\begin{proof}
From~\cite[Theorem 3]{Shao:2017:PDB}, if the matrix $\SignBSE H$ is positive definite, $H$ can be diagonalized as in~\eqref{eq:bse-diagonalization} and the converse is also true.
Then,
\[
\SignBSE q(H)=
\SignBSE q\left(
\begin{bmatrix}
  \XX & \oXY \\
  \XY & \oXX
\end{bmatrix}
\begin{bmatrix}
  \Lambda_+ &  \\
   & -\Lambda_+
\end{bmatrix}
\begin{bmatrix}
   \XX & -\oXY \\
  -\XY & \oXX
\end{bmatrix}^*
\right).
\]
Since
$
\begin{bmatrix}
   \XX & -\oXY \\
  -\XY & \oXX
\end{bmatrix}^*
\begin{bmatrix}
  \XX & \oXY \\
  \XY & \oXX
\end{bmatrix}
= I_{2m},
$
 from the definition of matrix function we know that
$$
\SignBSE q(H)=
\SignBSE
\begin{bmatrix}
  \XX & \oXY \\
  \XY & \oXX
\end{bmatrix}
q\left(
\begin{bmatrix}
  \Lambda_+ &  \\
   & -\Lambda_+
\end{bmatrix}
\right)
\begin{bmatrix}
   \XX & -\oXY \\
  -\XY & \oXX
\end{bmatrix}^*.
$$
If $\lambda_i>0$ (resp.~$\lambda_i<0$), $i=1,\dots,m$, then $q(\lambda_i)>0$ (resp.~$q(\lambda_i)<0$) because of the positivity property of the kernel~\cite{Weisse:2006:KPM}. Therefore, $\tilde{\Lambda}_+=q(\Lambda_+)$ is positive, and
$$
\SignBSE q(H)=
\SignBSE
\begin{bmatrix}
  \XX & \oXY \\
  \XY & \oXX
\end{bmatrix}
\begin{bmatrix}
  \tilde{\Lambda}_+ &  \\
   & -\tilde{\Lambda}_+
\end{bmatrix}
\begin{bmatrix}
   \XX & -\oXY \\
  -\XY & \oXX
\end{bmatrix}^*.
$$
Then we operate with this expression to obtain
$$
\SignBSE q(H)=
\begin{bmatrix}
   \XX & -\oXY \\
  -\XY & \oXX
\end{bmatrix}
\begin{bmatrix}
  \tilde{\Lambda}_+ &  \\
   & \tilde{\Lambda}_+
\end{bmatrix}
\begin{bmatrix}
   \XX & -\oXY \\
  -\XY & \oXX
\end{bmatrix}^*.
$$
Using Sylvester's law of inertia (*-congruence), $\SignBSE q(H)$ is positive definite.
\end{proof}
As we can see in~\cref{tab:damping}, both the Fejer and Jackson kernels satisfy the positivity property and are suitable for the filter.

\section{A Chebyshev filtered subspace iteration for BSE}\label{sec:subspace-iteration}

First, we will discuss how the subspace iteration method can be adapted for the definite BSE matrix $H$. Subsequently, in~\cref{sec:algorithm}, we will combine it with the structure-preserving filter.

The subspace iteration method described in~\cref{sec:si} relies on the well-known Rayleigh-Ritz procedure, that realizes an orthogonal projection onto a subspace $\mathcal{V}$ obtained by successively applying matrix $A$ to an initial subspace. More precisely, eigenvector approximations are selected from such subspace, $\tilde{x}\in\mathcal{V}$, in a way that the Galerkin condition is satisfied, i.e., the corresponding residual is orthogonal to the same subspace, $r\perp\mathcal{V}$, where $r=A\tilde{x}-\tilde{\lambda}\tilde{x}$ and $(\tilde{\lambda},\tilde{x})$ is the Ritz pair. Let $Q\in\mathbb{C}^{m\times k}$ be an orthogonal basis of $\mathcal{V}$, then $\tilde{x}=Qw$ for some $w\in\mathbb{C}^k$, and the Galerkin condition is expressed as $Q^*r=0$, resulting in the projected eigenproblem $Q^*AQw=\tilde{\lambda}w$.

In order to derive a structure-preserving version of subspace iteration for a BSE matrix $H$, we consider an oblique projection, where there are two different subspaces. In particular, the approximations $\tilde{x}\in\mathcal{V}$ satisfy the Petrov--Galerkin condition $r\perp\SignBSE\mathcal{V}$. In this case, the bases of the left and right subspaces, $\SignBSE Q$ and $Q$, respectively, with $Q\in\mathbb{C}^{2m\times 2k}$, must be bi-orthogonal,
\begin{equation}\label{eq:biorth}
Q^*\SignBSE Q=\SignOutBSE,
\end{equation}
where $\SignOutBSE$ is a signature matrix, or, equivalently, the $Q$ basis must be $(\SignBSE,\SignOutBSE)$-orthogonal, i.e., its columns must be orthogonal with respect to the indefinite inner product induced by the signature matrix $\SignBSE$. The Petrov--Galerkin condition then reads
\begin{equation}\label{eq:petrov-galerkin}
Q^*\SignBSE HQw=\tilde{\lambda}Q^*\SignBSE Qw,
\end{equation}
or
\begin{equation}\label{eq:projected}
Mw=\tilde{\lambda}\SignOutBSE w,
\end{equation}
with the Hermitian matrix $M=Q^*\hat{H}Q$.

Additionally, our method will converge to a partial diagonalization where the eigenvectors will have structure~\eqref{eq:bse-diagonalization}, so the sequence of subspaces generated by the algorithm will have bases with structure $Q=
\left[\begin{smallmatrix}
  \QX & \oQY \\
  \QY & \oQX
\end{smallmatrix}\right],
$ 
orthonormalized as in~\eqref{eq:biorth}, with $\SignOutBSE=
\left[\begin{smallmatrix}
  I_k &  \\
   & -I_k
\end{smallmatrix}\right]
$. We will discuss the required orthonormalization in~\cref{sec:orthogonality}.

Now, considering both constraints, the generalized eigenproblem~\eqref{eq:projected} can be written as
\begin{equation}
\begin{bmatrix}
  \MX & \MY \\
  \oMY & \oMX
\end{bmatrix}
w = \tilde{\lambda}
\begin{bmatrix}
  I_k &  \\
   & -I_k
\end{bmatrix}
w,
\end{equation}
which is equivalent to a standard eigenproblem with a matrix with BSE structure.

In other words, we derive the structure-preserving subspace iteration from viewing the BSE problem as a Hermitian positive-definite pencil $(\hat{H},\SignBSE)$ with an indefinite right-hand side matrix $\SignBSE$, which is projected to a small eigenproblem $(M,\SignOutBSE)$ with the same structure. A practical implementation of this procedure will require particular attention to the $(\SignBSE,\SignOutBSE)$-orthogonality of the basis as well as preserving the structure in all steps of the computation, e.g., when forming $M$.

\subsection{Orthogonalization} \label{sec:orthogonality}

The basis of the subspace $\mathcal{V}$ must be built so that it satisfies the bi-orthogonality condition~\eqref{eq:biorth}.
For this purpose, we employ the SVQB method~\cite{Stathopoulos:2002:BOP} adapted for the case of an indefinite inner product, $Q^*\SignIn Q = \SignOut$, where $\SignIn$, $\SignOut$ are two signature matrices of appropriate size. More precisely, for the specific case of the BSE matrix we propose a block version where $\SignIn = \SignBSE$ and $\SignOut = \SignOutBSE$. These algorithms have also been described recently by Shan and Shao~\cite{Shan:2026:SPL} in the context of a structure-preserving BSE version of the LOBPCG algorithm.

The SVQB method computes $X=QB$ in order to obtain an orthonormal basis $Q$, which is derived from the right singular vectors of $X$. Note the similarity to a QR decomposition with the difference that $B$ is a full matrix. However, since we are only interested in the basis $Q$, the shape of $B$ is irrelevant. This method is computationally efficient due to its use of BLAS-3 operations, and lower synchronization requirements in a parallel environment.

We use a generalized version of this algorithm to obtain a $(\SignIn ,\SignOut)$-orthogonal basis from the right singular vectors of a hyperbolic singular value decomposition~\cite{Onn:1991:HSV} of $X$,
\begin{equation}\label{eq:hsvd}
X=\SVDU \SVDVal \SVDV^*, \qquad \SVDU^*\SignIn U = \SignOut,\; V^*V=I,
\end{equation}
where we assume the columns in $X$ to be $\SignIn$-normalized.
The hyperbolic singular values $\SVDVal$ and right hyperbolic singular vectors $\SVDV$ are obtained from the eigendecomposition of $X^*\SignIn X$, so that $X^*\SignIn X=V\SVDVal^*\SVDU^*\SignIn\SVDU\SVDVal\SVDV^*=V\SVDVal^*\SignOut\SVDVal\SVDV^*$. This eigendecomposition is then
\begin{equation}
SV = V\SignOut\Lambda, \qquad S = X^*\SignIn X, \qquad \Lambda=\SVDVal^2.
\end{equation}
The $(\SignIn ,\SignOut)$-orthogonalized basis $Q$ is computed as $Q=U=X\SVDV\Lambda^{-1/2}$, satisfying $\text{span}(Q)=\text{span}(X)$, and $Q^*\SignIn Q = \SignOut$.

If the columns of $X$ are not $\SignIn$-normalized, the normalization is performed as in the SVQB method, taking advantage of the fact that $D = \text{diag}(S)$ contains the column norms of $X$.

This algorithm, which we will refer to as SVQBI (\cref{alg:svqbi}), can be adapted to the BSE matrix (\cref{alg:svqbi-bse}). We build $S'$ with $\SignIn=\SignBSE$ as

\[
\begin{bmatrix}
\SX' & -\oSY'\\
\SY' & -\oSX'
\end{bmatrix}
=
\begin{bmatrix}
\XX & \oXY \\
\XY & \oXX
\end{bmatrix}^* 
\SignBSE
\begin{bmatrix}
\XX & \oXY \\
\XY & \oXX
\end{bmatrix}.
\]
$S_1'$ is Hermitian by construction. Then 
$|\text{diag}(S')|=
\left[\begin{smallmatrix}
|\text{diag}(\SX')| & \\
& |\text{diag}(\SX')| 
\end{smallmatrix}\right]
=
\left[\begin{smallmatrix}
|D_1| & \\
& |D_1| 
\end{smallmatrix}\right],
$
where $D_1$ is real. We implicitly work with the normalized vector basis $XD^{-1/2}$ by scaling $S'$ to obtain
\[
\begin{bmatrix}
\SX & -\oSY\\
\SY & -\oSX
\end{bmatrix}
=
\begin{bmatrix}
|D_1|^{-1/2} & \\
& |D_1|^{-1/2} 
\end{bmatrix}
\begin{bmatrix}
\SX' & -\oSY'\\
\SY' & -\oSX'
\end{bmatrix}
\begin{bmatrix}
|D_1|^{-1/2} & \\
& |D_1|^{-1/2} 
\end{bmatrix}.
\]

The structure of $S$ is similar to a BSE structure, but with the difference that $S_2$ is skew-symmetric and, therefore, $S$ is Hermitian. We make the choice to solve this problem as Hermitian, although the structure of the matrix would allow to create a specific structure-preserving method. In fact, $S$ is $Z$-skew-adjoint with respect to the scalar product in its bilinear form ${\langle x,y\rangle}_Z=x^TZy$, with $Z=\left[\begin{smallmatrix}0 & I\\I & 0\end{smallmatrix}\right]$, which implies that $S$ forms a Lie algebra $\mathbb{L}$. Therefore, and according to Mackey \textit{et al.}~\cite{Mackey:2006:SFS}, its eigenvalues appear in pairs $(\lambda,-\lambda)$, which are real because $S$ is Hermitian. We now consider the eigenvectors of $S$. Using $S^T=-ZSZ$ we obtain that if $Sx=\lambda x$, then ${(Zx)}^TS=-\lambda{(Zx)}^T$. Additionally, $x^*S=\lambda x^*$ because of the Hermicity of $S$. Using these properties, $S$ can be diagonalized as

\[
\begin{bmatrix}\SX & -\oSY \\ \SY & -\oSX\end{bmatrix}
\begin{bmatrix}\VX & \oVY \\ \VY & \oVX \end{bmatrix}
=
\begin{bmatrix}\VX & \oVY \\ \VY & \oVX \end{bmatrix}
\begin{bmatrix}\Lambda_+ & \\ & -\Lambda_+\end{bmatrix},
\]
where the signs of the eigenvalues can be extracted as $\SignOutBSE$. \Cref{alg:svqbi-bse} does not give $\SignOutBSE$ as an output because this matrix is already known beforehand. 

Lastly, we obtain the $(\SignBSE,\SignOutBSE)$-orthogonalized basis as
\[
\begin{bmatrix}
\QX & \oQY \\
\QY & \oQX
\end{bmatrix}
=
\begin{bmatrix}
\XX & \oXY \\
\XY & \oXX
\end{bmatrix}
\begin{bmatrix}
|D_1|^{-1/2} & \\
& |D_1|^{-1/2} 
\end{bmatrix}
\begin{bmatrix}
\VX & \oVY \\
\VY & \oVX
\end{bmatrix}
\begin{bmatrix}
\Lambda_+^{-1/2} & \\ 
& -\Lambda_+^{-1/2}
\end{bmatrix}.
\]
Several iterations of SVQBI-BSE may be needed to obtain the $(\SignBSE,\SignOutBSE)$-orthogonalized basis in finite precision, as pointed out by the authors of the SVQB method.

We believe it would be possible to adapt the LDLIQR2 algorithm~\cite{Benner:2022:SEC} in a similar manner to the BSE matrix case too, as an alternative.

\begin{algorithm}[t]
  \caption{SVQBI}
  \label{alg:svqbi}
  \begin{algorithmic}[1]
    \REQUIRE A vector basis $X$, signature $\SignIn$
    \ENSURE $(\SignIn ,\SignOut)$-orthonormal basis $Q$, signature $\SignOut$
		\STATE $S'=X^* \SignIn X$
		\STATE Scale $S=|D|^{-1/2}S'|D|^{-1/2}$ with $D=\text{diag}(S')$
		\STATE Compute eigendecomposition $SV=V\Lambda$
		\STATE $Q=X|D|^{-1/2}V|\Lambda|^{-1/2}$
		\STATE $\SignOut = \text{sign}(\Lambda)$
  \end{algorithmic}
\end{algorithm}

\begin{algorithm}[t]
  \caption{SVQBI-BSE}
  \label{alg:svqbi-bse}
  \begin{algorithmic}[1]
    \REQUIRE Vector basis $\left[\begin{smallmatrix}\XX \\ \XY\end{smallmatrix}\right]$ 
    \ENSURE $(\SignBSE,\SignOutBSE)$-orthonormal basis $\left[\begin{smallmatrix}\QX \\ \QY\end{smallmatrix}\right]$
		\STATE $\SX'=\XX^*\XX-\XY^*\XY$
		\STATE $\SY'=\oXY^*\XX-\oXX^*\XY$
		\STATE Scale $\SX=|D_1|^{-1/2}\SX'|D_1|^{-1/2}$ with $D_1=\text{diag}(\SX')$
		\STATE Scale $\SY=|D_1|^{-1/2}\SY'|D_1|^{-1/2}$ with $D_1=\text{diag}(\SX')$
		\STATE Compute eigendecomposition $\left[\begin{smallmatrix}\SX & -\oSY \\ \SY & -\oSX\end{smallmatrix}\right] \left[\begin{smallmatrix}\VX &\oVY \\ \VY &\oVX\end{smallmatrix}\right] = \left[\begin{smallmatrix}\VX &\oVY\\ \VY&\oVX\end{smallmatrix}\right] \left[\begin{smallmatrix}\Lambda_+ & \\ & -\Lambda_+\end{smallmatrix}\right]$
		\STATE $\QX=\XX|D_1|^{-1/2}\VX\Lambda_+^{-1/2} + \oXY|D_1|^{-1/2}\VY\Lambda_+^{-1/2}$
		\STATE $\QY=\XY|D_1|^{-1/2}\VX\Lambda_+^{-1/2} + \oXX|D_1|^{-1/2}\VY\Lambda_+^{-1/2}$
  \end{algorithmic}
\end{algorithm}

\subsection{Description of the algorithm} \label{sec:algorithm}

We will now address the integration of our polynomial filter within this method.
In this case, the subspace $\mathcal{V}$, with basis $Q$, will be obtained by simultaneous iteration with $q(H)$. 
The Rayleigh-Ritz procedure will then be performed as previously discussed, so that $r\perp\SignBSE\mathcal{V}$, with $r=HQw-\tilde{\lambda}Qw$. This process extracts approximations of the eigenvalues and eigenvectors of $H$ from $\mathcal{V}$ that correspond to those in the wanted subinterval. 
Note that these are approximate eigenvectors of both $H$ and $q(H)$, since the eigenvectors of $H$ are also eigenvectors of $q(H)$~\cite{Higham:2008:FMT}.
The approximate eigenvalues also indicate whether they lie within the specified subinterval, since we do not know, a priori, how many eigenvalues the subinterval contains. In fact, it is possible to underestimate the number of column vectors required to compute all the eigenpairs inside the subinterval.

Next, we will derive the method, emphasizing how the structure of the matrices is preserved throughout the steps, which allows us to compute only half of the blocks of the matrices involved, as the other half is handled implicitly. The resulting process, described by means of \cref{alg:chebfsi-bse-blocks}, consists of the following steps:

\begin{enumerate}
\item Initial subspace: Although a random subspace is a sufficiently good starting subspace for the general subspace iteration method, in our case, the initial subspace must have the structure
$
\left[\begin{smallmatrix}
\XXi & -\oXYi \\
\XYi & -\oXXi
\end{smallmatrix}\right].
$
This is handled implicitly by providing randomly generated blocks $\XXi,\XYi$ as the input to the method.

\item Vector simultaneous iteration: An updated basis is computed as $\pHX=q(H)X$ in line~\ref{algl:clenshaw-bsefsi}. This is done using an adaptation of the Clenshaw algorithm, as will be explained in \cref{sec:clenshaw}. $\pHX$ will have structure 
$
\left[\begin{smallmatrix}
\pHXX & \opHXY \\
\pHXY & \opHXX
\end{smallmatrix}\right].
$

\item Orthonormalization:
A basis of $\mathcal{V}$, $Q$, is obtained as a $(\SignBSE,\SignOutBSE)$-orthonormalized basis spanning the same subspace as $\pHX$, using~\cref{alg:svqbi-bse}, as explained in~\cref{sec:orthogonality}.

\item Rayleigh-Ritz projection: 
A first step computes $\HQ=HQ$, which can be reused later for the computation of the residual. Line~\ref{algl:product-1} of \cref{alg:chebfsi-bse-blocks} performs the product with the blocks of $H$ and $Q$,
\begin{equation}\label{eq:product-updated-basis}
\begin{bmatrix}
\HQX & -\oHQY \\
\HQY & -\oHQX
\end{bmatrix}
=
\begin{bmatrix}
R & C \\
-\overline{C} &-\overline{R}
\end{bmatrix}
\begin{bmatrix}
\QX & \oQY \\
\QY & \oQX
\end{bmatrix}.
\end{equation}

An important difference with the generic subspace iteration method is that, due to the product structure, only half of the resulting matrix needs to be computed. One can also avoid forming the matrices on the right of \eqref{eq:product-updated-basis} by formulating the operation in terms of four block multiplications.
However, computing these block multiplications sequentially, one after the other, can lead to worse performance than computing \eqref{eq:product-updated-basis} directly without taking into account any structure.
One can trade off memory for speed, and compute $\HQX,\HQY$ at the same time from the relation
\[ 
\begin{bmatrix}
\HQX & -\oHQY
\end{bmatrix}
= R
\begin{bmatrix}
\QX & \oQY
\end{bmatrix}
+C
\begin{bmatrix}
\QY & \oQX
\end{bmatrix}.
\]
We will use the same technique when applying the filter in~\cref{sec:clenshaw}, where it will be even more relevant, since that is the most expensive step of the algorithm. If no filter were applied, product \eqref{eq:product-updated-basis} would be the most expensive operation. 

Then, lines~\ref{algl:rr-construct} to~\ref{algl:rr-solve} build and solve the projected eigenproblem~\eqref{eq:petrov-galerkin}, employing the previously computed $\HQ$ when constructing $\SignOutBSE^{-1}M$, as
\begin{equation}\label{eq:rayleigh-ritz-bse}
\SignOutBSE^{-1}M=\SignOutBSE
\begin{bmatrix}
\QX^* & \QY^* \\
\oQY^* & \oQX^*
\end{bmatrix}
\SignBSE
\begin{bmatrix}
\HQX & -\oHQY \\
\HQY & -\oHQX
\end{bmatrix}
=
\begin{bmatrix}
  \MX & \MY \\
  -\oMY & -\oMX
\end{bmatrix}.
\end{equation}
The eigendecomposition of $\SignOutBSE^{-1}M$, 
\begin{equation}
\begin{bmatrix}
  \MX & \MY \\
  -\oMY & -\oMX
\end{bmatrix}
\begin{bmatrix}
\WX & -\oWY \\
\WY & -\oWX
\end{bmatrix}
=
\begin{bmatrix}
\WX & -\oWY \\
\WY & -\oWX
\end{bmatrix}
\begin{bmatrix}
\Lambda_+ &  \\
 & -\Lambda_+ 
\end{bmatrix},
\end{equation}
is computed using the dense method in~\cite[Algorithm~1]{Shao:2016:SPP}, which converts the projected BSE eigenvalue problem into an equivalent real Hamiltonian eigenvalue problem.
Then, the bases $Q$ and $\HQ$ are updated as
\begin{gather}
\begin{bmatrix}
\XX & -\oXY \\
\XY & -\oXX
\end{bmatrix}
=
\begin{bmatrix}
\QX & \oQY \\
\QY & \oQX
\end{bmatrix}
\begin{bmatrix}
\WX & -\oWY \\
\WY & -\oWX
\end{bmatrix}, \\
\begin{bmatrix}
\HXX & \oHXY \\
\HXY & \oHXX
\end{bmatrix}
=
\begin{bmatrix}
\HQX & -\oHQY \\
\HQY & -\oHQX
\end{bmatrix}
\begin{bmatrix}
\WX & -\oWY \\
\WY & -\oWX
\end{bmatrix}.
\end{gather}
\item Residual: The residual $\mathcal{R} = HQW - QW\Lambda$ is obtained using the previously computed $\HX = HQW$ and $X = QW$. This is done in lines~\ref{algl:residual-1} and~\ref{algl:residual-2} as
$$
\begin{bmatrix}
\RX & \oRY \\
\RY & \oRX
\end{bmatrix}
=
\begin{bmatrix}
\HXX & \oHXY \\
\HXY & \oHXX
\end{bmatrix}
-
\begin{bmatrix}
\XX & -\oXY \\
\XY & -\oXX
\end{bmatrix}
\begin{bmatrix}
\Lambda_+ &  \\
 & -\Lambda_+ 
\end{bmatrix}.
$$

\item Convergence check: 
To consider that an eigenvalue in $\Lambda$ has converged, the norm of the corresponding column of the residual must be below a given tolerance. Note that a pair of eigenvalues $(\lambda,-\lambda)$ converge simultaneously, since columns $j$ and $j+k$ of $\mathcal{R}$, $j=1,\dots,k$, have the same norm. Sometimes spurious values can appear, it is easy to identify them by their large residual norms. The strategy we employ to remove them follows the approach proposed in~\cite{Xu:2026:APF}, using a threshold computed in each iteration.
\end{enumerate}

\begin{algorithm}[t]
  \caption{Chebyshev filtered subspace iteration for BSE}
  \label{alg:chebfsi-bse-blocks}
  \begin{algorithmic}[1]
    \REQUIRE A definite $2m\times 2m$ BSE matrix $H$ with structure \eqref{eq:bse-structure} and spectral radius $\lambda_{\text{max}}$; blocks from the initial subspace basis $\XXi,\XYi$; coefficients $c_0,\dots,c_N$; kernel coefficients $g_0,\dots,g_N$
    \ENSURE Selected approximate eigenvalues $\Lambda_+$ and eigenvectors $\left[\begin{smallmatrix}\XX \\ \XY\end{smallmatrix}\right]$
        \STATE Set $\XX=\XXi, \XY=\XYi$
		\WHILE{not converged}
			\STATE $\pHXX,\pHXY = \text{clenshaw\_bse\_fsi}(R,C,\XX,\XY,c,g,\lambda_{\text{max}})$ \label{algl:clenshaw-bsefsi}
			\STATE $\QX,\QY = \text{svqbi\_bse}(\pHXX,\pHXY)$ \label{algl:svqbi-bse}
 			\STATE $ \begin{bmatrix} \HQX & \HQYaux \end{bmatrix} = R \begin{bmatrix} \QX & \oQY \end{bmatrix} +C \begin{bmatrix} \QY & \oQX \end{bmatrix}$ \label{algl:product-1}
 			\STATE $\HQY = -\oHQYaux$
			\STATE $\MX= \QX^*\HQX-\QY^*\HQY$ \label{algl:rr-construct}
			\STATE $\MY= -\QX^*\oHQY+\QY^*\oHQX$
			\STATE $\WX,\WY,\Lambda_+ = \text{dense\_bse}(\MX,\MY)$ \label{algl:rr-solve}
			\STATE $\HXX = \HQX \WX -\oHQY\WY$
			\STATE $\HXY = \HQY \WX-\oHQX\WY$
			\STATE $\XX = \QX \WX -\oQY\WY$
			\STATE $\XY = \QY \WX-\oQX\WY$
			\STATE $\RX = \HXX - \XX \Lambda_+$ \label{algl:residual-1}
			\STATE $\RY = \HXY - \XY \Lambda_+$ \label{algl:residual-2}
			\STATE \text{Check convergence}
		\ENDWHILE
  \end{algorithmic}
\end{algorithm}

\subsection{Locking}
Locking is a commonly used technique to deflate the converged vectors, avoiding unnecessary operations. We perform it in our method in a similar manner as in~\cite{Xu:2026:APF}. We do not include the handling of the locked vectors in~\cref{alg:chebfsi-bse-blocks} for simplicity, but we demonstrate in this subsection how the technique can be adapted to the BSE case.

We split the vector basis $Q$ in its locked and non-locked parts, and permute its columns so that
\[
Q=
\begin{bmatrix}
\LX & \NX & \oLY & \oNY \\
\LY & \NY & \oLX & \oNX 
\end{bmatrix}
=
\begin{bmatrix}
\LX & \oLY & \NX & \oNY \\
\LY & \oLX & \NY & \oNX 
\end{bmatrix}
P=
\begin{bmatrix}
L & N \\
\end{bmatrix}
P,
\]
where $L$ are the $\sizeL$ locked columns, that have converged in previous iterations.
Note that \eqref{eq:biorth} implies that $L^*\SignBSE L=\SignOutBSEL$, with $\SignOutBSEL=\left[\begin{smallmatrix}I_{\sizeL} \\ & -I_{\sizeL}\end{smallmatrix}\right]$, and $N^*\SignBSE L=0$.
Then, the Rayleigh-Ritz projected matrix $M$ can be written as
\[
M = P^*\begin{bmatrix}L & N\end{bmatrix}^*\SignBSE H\begin{bmatrix}L & N\end{bmatrix}P
=
P^*
\begin{bmatrix}
L^*\SignBSE HL & L^*\SignBSE HN \\
N^*\SignBSE HL & N^*\SignBSE HN \\
\end{bmatrix}
P.
\]

In order to apply locking, we assume that the converged columns $L$ verify $HL=L\Lambda_L$, where $\Lambda_L=\left[\begin{smallmatrix}\Lambda_{+L} \\ & -\Lambda_{+L}\end{smallmatrix}\right]$. Since $L^*\SignBSE L=\SignOutBSEL$, we obtain that $L^*\SignBSE HL=\SignOutBSEL \Lambda_L$.

Since $N^*\SignBSE L=0$, we also have that $N^*\SignBSE HL= N^*\SignBSE L\Lambda = 0$ and $(L^*\SignBSE HN)^*=N^*(\SignBSE H)^*L = N^*\SignBSE HL = 0.$

Then, we can substitute in~\eqref{eq:rayleigh-ritz-bse},
\[
\SignOutBSE^{-1}M=
\SignOutBSE P^*
\begin{bmatrix}
\SignOutBSEL \Lambda_L &  \\
       & \tilde{M} \\
\end{bmatrix}
P
=
P^*
\begin{bmatrix}
\Lambda_{+L} &           &      &      \\
          & -\Lambda_{+L} &      &      \\
          &           & \MX  & \MY  \\
          &           & -\oMY & -\oMX \\
\end{bmatrix}
P.
\]
We can exclude the columns in $L$ from the simultaneous iteration, Rayleigh-Ritz projection, and convergence check steps. However, they must be taken into account in the orthonormalization step. In particular, given $\pHX=\begin{bmatrix}L & \hat{N}\end{bmatrix}P$ resulting from line \ref{algl:clenshaw-bsefsi} of \cref{alg:chebfsi-bse-blocks}, a two-phase orthogonalization process is carried out in line \ref{algl:svqbi-bse} to obtain $N$ such that $\text{span}(\begin{bmatrix}L & N\end{bmatrix}) = \text{span}(\begin{bmatrix}L & \hat{N}\end{bmatrix})$ and $Q=\begin{bmatrix}L & N\end{bmatrix}P$ is $(\SignBSE,\SignOutBSE)$-orthogonal.
First, an external phase orthogonalizes $\hat{N}$ with respect to $L$, by computing $\tilde{N} = \hat{N}-L\SignOutBSE L^*\SignBSE \hat{N}$, which can be seen as one step of indefinite block Gram--Schmidt. Then, an internal phase orthogonalizes the columns of $\tilde{N}$ by applying~\cref{alg:svqbi-bse} to them.

Each of the two phases can produce loss of orthogonality in the other, so several iterations may be needed.

\section{Clenshaw for BSE}\label{sec:clenshaw}

The Clenshaw algorithm can be used to compute the product of a truncated Chebyshev series times a vector. In this section, we will consider that the matrices have not been scaled, since the scaling will be done inside the algorithm. 
We will denote the scaled matrix $\tilde{A} \in \mathbb{C}^{m \times m}$ and include the scaling in~\eqref{eq:clenshaw-not-scaled} to obtain
\begin{gather*}
p(\tilde{A})x=
  \frac{1}{\rho}Ay^{(1)} - \frac{\sigma}{\rho}y^{(1)} - y^{(2)} +g_0c_0x,\\
y^{(n)} =
  \frac{2}{\rho}Ay^{(n+1)} - \frac{2\sigma}{\rho}y^{(n+1)}- y^{(n+2)} +g_nc_nx,\\
y^{(N+2)} =y^{(N+1)} =0,
\end{gather*}
with $\rho$ and $\sigma$ from~\eqref{eq:scaling}.

A Bethe--Salpeter version of the algorithm for a scaled matrix $\tilde{H}$ and the structured filter $q$ would have two differences with respect to the general one: (i) the scaling of the matrix is simpler, as in~\eqref{eq:bse-scaling}, allowing for the removal of a sum; and (ii) the term $g_0c_0x$ is zero, as are the terms $g_nc_nx$ for even $n$, according to \cref{lma:coeffs}. This leads to
\begin{gather*}
q(\tilde{H})x=\frac{1}{\lambda_{\text{max}}}Hy^{(1)} - y^{(2)},\\
y^{(n)} = \begin{cases}
\frac{2}{\lambda_{\text{max}}}Hy^{(n+1)} - {y}^{(n+2)} + g_nc_nx, & n=2k+1, \\
\frac{2}{\lambda_{\text{max}}}Hy^{(n+1)} - {y}^{(n+2)},        & n=2k,
\end{cases}\\
y^{(N+2)} =y^{(N+1)} =0.
\end{gather*}

Next, we will see how we can adapt the Clenshaw method to our structure-preserving subspace iteration for the BSE, and to the Lanczos-based method from Shao \textit{et al.}~\cite{Shao:2018:SPL}.

\subsection{Clenshaw for the BSE subspace iteration}

The Clenshaw algorithm can also be applied to a matrix-matrix product with $q(\tilde{H})$. In fact, in line~\ref{algl:clenshaw-bsefsi} of~\cref{alg:chebfsi-bse-blocks} we would apply it to compute the product $q(\tilde{H})X$ as
\begin{gather*}
q(\tilde{H})\begin{bmatrix}\XX & \oXY \\ \XY & \oXX\end{bmatrix}=\frac{1}{\lambda_{\text{max}}}
H\begin{bmatrix}\YnX{1} & \oYnY{1} \\ \YnY{1} & \oYnX{1}\end{bmatrix} 
- \begin{bmatrix}\YnX{2} & -\oYnY{2} \\ \YnY{2} & -\oYnX{2}\end{bmatrix}
=\begin{bmatrix}\pHXX & -\opHXY \\ \pHXY & -\opHXX\end{bmatrix},
\\
Y^{(n)} = \begin{cases}
\frac{2}{\lambda_{\text{max}}}H\begin{bmatrix}\YVX & -\oYVY \\ \YVY & -\oYVX\end{bmatrix}
 - \begin{bmatrix}\YWX & \oYWY \\ \YWY & \oYWX\end{bmatrix} 
+ g_nc_n\begin{bmatrix}\XX & \oXY \\ \XY & \oXX\end{bmatrix}, & n=2k+1, \\
\frac{2}{\lambda_{\text{max}}}H\begin{bmatrix}\YVX & \oYVY \\ \YVY & \oYVX\end{bmatrix} 
- \begin{bmatrix}\YWX & -\oYWY \\ \YWY & -\oYWX\end{bmatrix}, & n=2k,
\end{cases}\\
Y^{(N+2)} =Y^{(N+1)} =0.
\end{gather*}

We can see that the structure is preserved similarly to~\eqref{eq:product-updated-basis}. This leads to~\cref{alg:clenshaw-chebfsi}, where the Clenshaw method is only applied to $\left[\begin{smallmatrix}\XX \\ \XY\end{smallmatrix}\right]$, and the right half of $\pHX$ is obtained implicitly. Products with $H$ (Line \ref{algl:clenshaw-chebfsi-Hprod}) can be decomposed into smaller products with the matrix blocks, although doing so can lead to poor performance, as was the case with the product with $H$ in~\eqref{eq:product-updated-basis}.

\begin{algorithm}[t]
  \caption{Clenshaw for the subspace iteration BSE}
  \label{alg:clenshaw-chebfsi}
  \begin{algorithmic}[1]
    \REQUIRE A definite $2m\times 2m$ BSE matrix $H$ with structure \eqref{eq:bse-structure} and spectrum in range $\{-\lambda_{\text{max}},\lambda_{\text{max}}\}$; matrices $\XX$ and $\XY$; polynomial degree $N$, polynomial coefficients $c_0,\dots,c_N$ and damping  coefficients $g_0,\dots,g_N$
    \ENSURE Matrices $\pHXX$ and $\pHXY$ such that $\left[\begin{smallmatrix}\pHXX\\\pHXY\end{smallmatrix}\right]=q(\tilde{H})\left[\begin{smallmatrix}\XX\\\XY\end{smallmatrix}\right]$
		\STATE $\begin{bmatrix}Y_1^{(N+1)} & Y_2^{(N+1)}\end{bmatrix} = \begin{bmatrix}0 & 0\end{bmatrix}$
		\STATE $\begin{bmatrix}Y_1^{(N)} & Y_2^{(N)}\end{bmatrix} = g_{N}c_{N}\begin{bmatrix}\XX & \XY\end{bmatrix}$
		\FOR{$n=N-1,\dots,0$}
			\STATE Let $s=\left\lbrace\begin{aligned}
				&2/\lambda_{\text{max}},\quad \text{ if } n>0\\
				&1/\lambda_{\text{max}},\quad \text{ if } n=0
			\end{aligned}\right.$
      \STATE $\begin{bmatrix}\YYX & \YYYaux\end{bmatrix} = R\begin{bmatrix}\YVX & \oYVY\end{bmatrix} + C\begin{bmatrix}\YVY & \oYVX\end{bmatrix}$\label{algl:clenshaw-chebfsi-Hprod}
      \STATE $\YYY \leftarrow -\oYYYaux$
      \STATE $\begin{bmatrix}\YYX & \YYY\end{bmatrix} \leftarrow s\begin{bmatrix}\YYX & \YYY\end{bmatrix} - \begin{bmatrix}\YWX & \YWY\end{bmatrix}$
      \IF {$n$ is odd}
        \STATE $\begin{bmatrix}\YYX & \YYY\end{bmatrix} \leftarrow  \begin{bmatrix}\YYX & \YYY\end{bmatrix}+ g_{n}c_{n}\begin{bmatrix}\XX & \XY\end{bmatrix}$
			\ENDIF
		\ENDFOR
      \STATE Set $\pHXX=\YZX, \pHXY=\YZY$
  \end{algorithmic}
\end{algorithm}

\subsection{Clenshaw for the BSE Lanczos method by Shao \textit{et al}}

As mentioned in \cref{sec:bse}, some BSE structure-preserving methods, e.g., \cite{Shao:2018:SPL}, work with vectors with the form $[u^T, \varsigma\bar{u}^T]^T$, with a sign $\varsigma=\pm 1$, which makes it possible to
operate with only the upper part of the vectors, using the property \cref{eq:H-times-u-conj-u}.

The Clenshaw algorithm can be adapted to this case, so that operation with the lower part of the vector is implicit. To do so, the algorithm needs to keep track of the lower part sign $\varsigma$ throughout the series of consecutive matrix-vector products.
In particular, given a vector $u$, the corresponding vector $v$ resulting from
\[
q(H)\begin{bmatrix}u \\ \varsigma\overline{u}\end{bmatrix}=\begin{bmatrix}v \\ -\varsigma\overline{v}\end{bmatrix},
\]
is computed by means of the recurrence
\begin{gather*}
v=\frac{1}{\lambda_{\text{max}}}(Ry^{(1)}+\varsigma_1 C\overline{y}^{(1)}) - y^{(2)},\\
y^{(n)} = 
\frac{2}{\lambda_{\text{max}}}Ry^{(n+1)}+\varsigma_{n+1}C\overline{y}^{(n+1)} - {y}^{(n+2)} + c_n u,\quad \varsigma_{n}=-\varsigma_{n+1},\qquad 1\le n\le N\\
y^{(N+2)} =y^{(N+1)} =0,\quad \varsigma_{N+1}=\varsigma.
\end{gather*}

This transformation is expected to have a positive impact on performance and is required in order to use the algorithm in combination with some structure-preserving iterative eigensolvers.

\section{Computational results}\label{sec:results}

The objective of this section is to validate the methods described in this article and present some computational results. For that purpose we will use two matrices from real applications. The first one is obtained from the Yambo code and corresponds to a simulation of a two-dimensional crystalline material (chromium triiodide). The second one comes from the study of the ground-state and excited-state calculations for a single strand DNA helix fragment with the TD-DFT method. The latter is interesting for us since it is possible to sparsify this matrix, i.e., drop matrix entries whose absolute value is below a given threshold. Even though the resulting matrix has a perturbed spectrum, in this case the obtained optical absorption spectrum matches reasonably well in the range of interest for the application. Additional details from these matrices are described in~\cref{tab:testproblems}. \cref{fig:spectrum} shows the distribution of the first 100 eigenvalues.

\begin{table}%
\caption{Description of the test problems used in the computational experiments: type of problem (sparse or dense), dimension $m$ of the $R$ and $C$ blocks, percentage of non-zero elements (nnz), largest eigenvalue $\lambda_{\text{max}}$, smallest positive eigenvalue $\lambda_1$, and average separation of the first 100 positive eigenvalues (sep).}
\label{tab:testproblems}
\begin{minipage}{\columnwidth}
\begin{center}
\begin{tabular}{llccccc}
\hline
name & type & $m$ & nnz & $\lambda_{\text{max}}$ & $\lambda_1$ & sep\\
\hline
\cri & dense  & 11520 & 100\% & 0.37646208 & $0.0518094226$ & $1.6\cdot 10^{-4}$\\
\dna & sparse & 12000 & 3.2\% & 1.46944408 & $0.1240817806$ & $1.2\cdot 10^{-3}$\\
\hline
\end{tabular}
\end{center}
\end{minipage}
\end{table}%

\begin{figure}[ht]
  \begin{center}
    \begin{tikzpicture}[xscale=0.9]
      \draw (0,0) -- (12.4,0);
      \foreach \x in \xallCrI \draw[gray,thin] (\x,-.25) -- +(0,.5);
      \draw[gray,thin] (0,-.25) -- +(0,.25);
      \node[below] at (0,-.4) {$5\cdot 10^{-2}$};
      \draw[thick,blue] (0,.6) -- ++(1.8093308239718,0)
                      foreach \x in \xincCrI { -| ++(\x,0.038462) }
                      -- ++(.1,0);
    \end{tikzpicture}
    \begin{tikzpicture}[xscale=0.8]
      \draw (0,0) -- (14,0);
      \foreach \x in \xallDNA \draw[gray,thin] (\x,-.25) -- +(0,.5);
      \draw[gray,thin] (0,-.25) -- +(0,.25);
      \node[below] at (0,-.4) {$1\cdot 10^{-1}$};
      \draw[thick,blue] (0,.6) -- ++(2.4081780613503,0)
                      foreach \x in \xincDNA { -| ++(\x,0.038462) }
                      -- ++(.1,0);
    \end{tikzpicture}
  \end{center}
  \caption{\label{fig:spectrum}Distribution of the first 100 positive eigenvalues of the \cri and \dna matrices, respectively. The blue line increments its height by a fixed amount when an eigenvalue is detected going from left to right on the distribution graph, to showcase the presence of eigenvalue clusters or multiple eigenvalues.}
\end{figure}
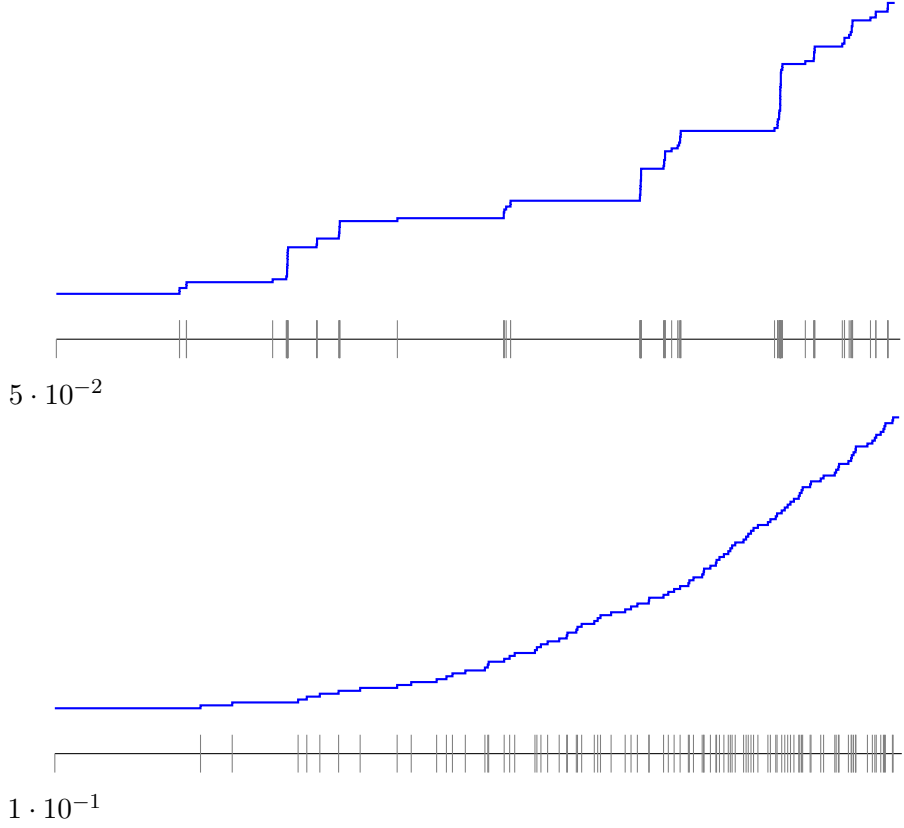

The experiments were carried out in MATLAB R2024a, using an AMD EPYC GENOA 9354 processor with two sockets with 32 physical cores each and hyperthreading disabled (64 threads total), at 3.8 GHz maximum frequency, and 384 GB of main memory (DDR5 at 4800 MHz).
The methods described are sequential but the MATLAB internal operations utilize the 64 available threads.

When using the polynomial filter, we are computing $\lambda_{\text{max}}$ using MATLAB's function \texttt{eigs}. The time for this computation is not included in the results. For the damping coefficients we use the Jackson kernel. We did not perform reorthogonalization in the SVQBI-BSE method in any of our experiments, and when employing the filter we always enabled locking and performed only one outer and one inner phase in the orthogonalization. The tolerance for the residual was $10^{-8}$ in all cases.

\subsection{Structure-preserving subspace iteration for BSE}
As a first experiment, we compare our structure-preserving BSE subspace iteration method (SI-BSE, \cref{alg:chebfsi-bse-blocks} without applying the filter) with the standard subspace iteration algorithm (SI, \cref{alg:subspace-iteration}) and analyze the performance gain.
In this experiment we do not employ any filtering, and therefore we are computing the largest eigenvalues. We compute 20 eigenpairs, for a range of initial subspace bases with varying number of column vectors (\textsf{ncv}), and compare both iterations and time. When using the structure-preserving method, we compute half of the eigenpairs (10), since we obtain the negative eigenpairs from the positive ones. In fact,~\cref{fig:si-vs-sibse} shows that the number of iterations is similar for the SI-BSE and SI with double number of column vectors. In practice, the gain in performance comes from avoiding the explicit computation of the negative part of the spectrum.

\begin{figure}
\label{fig:si-vs-sibse}
\centering
\begin{tikzpicture}[scale=0.7]
  \begin{axis}[
  title={\textsf{\cri}},
  ylabel={Iterations},
  xlabel={\textsf{ncv}},
  grid=major,
  ticklabel style={font=\small},
  legend style={draw=none, legend columns=1,font=\small,cells={anchor=west}}
  ]
  \addplot coordinates { 
    (30, 952 )
    (40, 802 )
    (50, 693 )
    (60, 606 )
    (70, 562 )
    (80, 466 )
  };
  \addplot coordinates { 
    (40, 1706)
    (50, 1351)
    (60, 971 )
    (70, 826 )
    (80, 765 )
    (90, 718 )
    (100,676 )
    (110,638 )
    (120,575 )
    (130,543 )
    (140,537 )
    (150,504 )
    (160,464 )
  };
  \legend{SI-BSE, SI}
  \end{axis}
\end{tikzpicture}
\begin{tikzpicture}[scale=0.7]
  \begin{axis}[
  title={\textsf{\cri}},
  ylabel={Time [s]},
  xlabel={\textsf{ncv}},
  grid=major,
  ticklabel style={font=\small},
  legend pos=north west,
  legend style={draw=none, legend columns=1,font=\small,cells={anchor=west}}
  ]
  \addplot coordinates { 
    (20, 164.9842)
    (30, 127.9034)
    (40, 136.4578)
    (50, 143.9874)
    (60, 147.603 )
    (70, 160.8682)
    (80, 157.5644)
  };
  \addplot coordinates { 
    (30, 233.2898)
    (40, 199.3322)
    (50, 184.0177)
    (60, 162.4767)
    (70, 168.626 )
    (80, 190.3679)
    (90, 190.0982)
    (100,228.426 )
    (110,242.5406)
    (120,234.0726)
    (130,234.6902)
    (140,243.5017)
    (150,273.4499)
    (160,267.2556)
  };
  \legend{SI-BSE, SI}
  \end{axis}
\end{tikzpicture}
\caption{The generic subspace iteration algorithm (SI) and our BSE structure-preserving version (SI-BSE) compared on iterations and time for computing the 20 largest eigenvalues in magnitude and their corresponding eigenvectors. The SI-BSE method computes only 10, the positive ones.}
\end{figure}
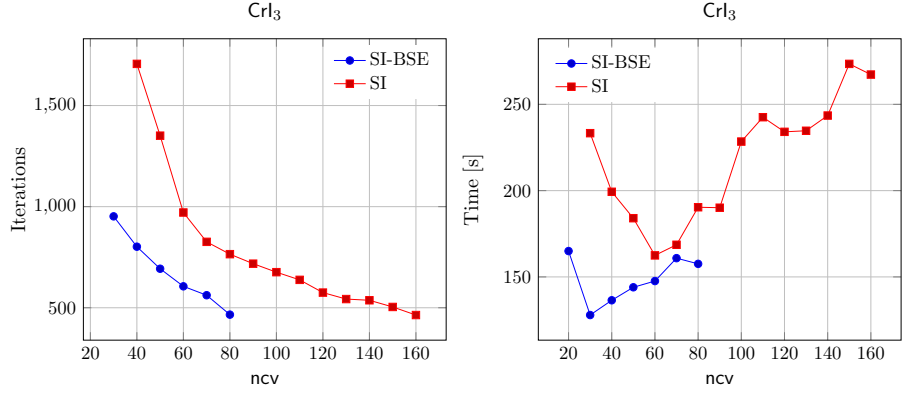

\subsection{Structure-preserving filter and parameters}
Our next experiment is a study of the convergence and performance of the filtered structure-preserving subspace iteration for BSE with varying polynomial degree ($N$) and number of column vectors (\textsf{ncv}). \cref{fig:filter-parameters-locking} shows the number of iterations and time for the two matrices employed, showcasing the compromise between reducing the number of iterations and performing products with larger matrices.

For the \cri matrix, 98.6\%  to 99.7\% of the time corresponded to applying the Clenshaw method, and inside it, 88.9\% to 92.1\% corresponded to the products with the blocks of $H$.
For the \dna matrix, those percentages were 98.2\%  to 99.7\% and 82.8\% to 91.8\% respectively. The cost of the product operations dominates the total cost. The number of products with $H$ performed (in blocks) is of the order of the degree times the number of iterations. Since those products have a high arithmetic intensity, the objective of a high performance implementation would be to optimize that step for GPU.

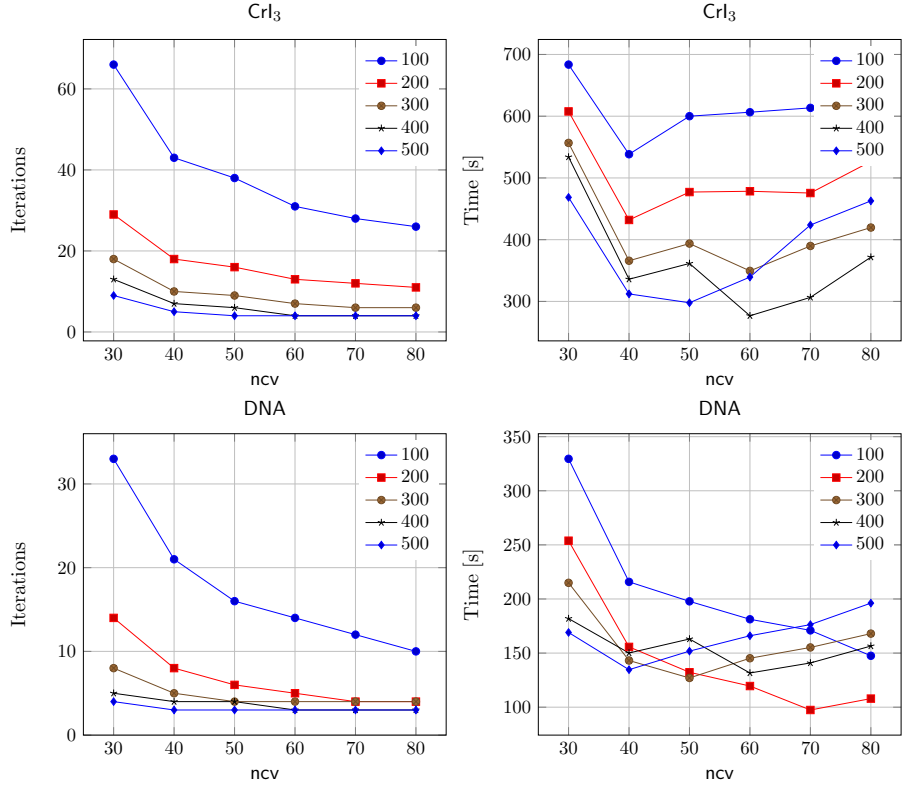
\begin{figure}
\label{fig:filter-parameters-locking}
\centering
\begin{tikzpicture}[scale=0.7]
  \begin{axis}[
  title={\textsf{\cri}},
  ylabel={Iterations},
  xlabel={\textsf{ncv}},
  grid=major,
  ticklabel style={font=\small},
  legend style={draw=none, legend columns=1,font=\small,cells={anchor=west}}
  ]
  \addplot coordinates { 
    (30,66)
    (40,43)
    (50,38)
    (60,31)
    (70,28)
    (80,26)
  };
  \addplot coordinates { 
    (30,29)
    (40,18)
    (50,16)
    (60,13)
    (70,12)
    (80,11)
  };
  \addplot coordinates { 
    (30,18)
    (40,10)
    (50,9 )
    (60,7 )
    (70,6 )
    (80,6 )
  };
  \addplot coordinates { 
    (30,13)
    (40,7 )
    (50,6 )
    (60,4 )
    (70,4 )
    (80,4 )
  };
  \addplot coordinates { 
    (30,9)
    (40,5 )
    (50,4 )
    (60,4 )
    (70,4 )
    (80,4 )
  };
  \legend{100,200,300,400,500}
  \end{axis}
\end{tikzpicture}
\begin{tikzpicture}[scale=0.7]
  \begin{axis}[
  title={\textsf{\cri}},
  ylabel={Time [s]},
  xlabel={\textsf{ncv}},
  grid=major,
  ticklabel style={font=\small},
  legend style={draw=none, legend columns=1,font=\small,cells={anchor=west}}
  ]
  \addplot coordinates { 
    (30,683.5891)
    (40,538.3689)
    (50,600.0319)
    (60,606.4333)
    (70,613.4888)
    (80,644.1393)
  };
  \addplot coordinates { 
    (30,607.7493)
    (40,431.9648)
    (50,477.1397)
    (60,478.4139)
    (70,475.5668)
    (80,526.7354)
  };
  \addplot coordinates { 
    (30,556.7192)
    (40,365.8786)
    (50,393.6963)
    (60,349.4309)
    (70,389.8217)
    (80,419.6537)
  };
  \addplot coordinates { 
    (30,533.8952)
    (40,335.9558)
    (50,361.4665)
    (60,276.7784)
    (70,306.5563)
    (80,371.6866)
  };
  \addplot coordinates { 
    (30,468.5206)
    (40,312.171 )
    (50,297.8307)
    (60,339.436 )
    (70,423.7584)
    (80,462.8062)
  };
  \legend{100, 200, 300, 400, 500}
  \end{axis}
\end{tikzpicture}
\begin{tikzpicture}[scale=0.7]
  \begin{axis}[
  title={\textsf{\dna}},
  ylabel={Iterations},
  xlabel={\textsf{ncv}},
  grid=major,
  ticklabel style={font=\small},
  legend style={draw=none, legend columns=1,font=\small,cells={anchor=west}}
  ]
  \addplot coordinates { 
    (30,33)
    (40,21)
    (50,16)
    (60,14)
    (70,12)
    (80,10)
  };
  \addplot coordinates { 
    (30,14)
    (40,8 )
    (50,6 )
    (60,5 )
    (70,4 )
    (80,4 )
  };
  \addplot coordinates { 
    (30,8 )
    (40,5 )
    (50,4 )
    (60,4 )
    (70,4 )
    (80,4 )
  };
  \addplot coordinates { 
    (30,5 )
    (40,4 )
    (50,4 )
    (60,3 )
    (70,3 )
    (80,3 )
  };
  \addplot coordinates { 
    (30,4 )
    (40,3 )
    (50,3 )
    (60,3 )
    (70,3 )
    (80,3 )
  };
  \legend{100,200,300,400,500}
  \end{axis}
\end{tikzpicture}
\begin{tikzpicture}[scale=0.7]
  \begin{axis}[
  title={\textsf{\dna}},
  ylabel={Time [s]},
  xlabel={\textsf{ncv}},
  grid=major,
  ticklabel style={font=\small},
  legend style={draw=none, legend columns=1,font=\small,cells={anchor=west}}
  ]
  \addplot coordinates { 
    (30,329.5281)
    (40,215.8216)
    (50,197.8059)
    (60,181.2502)
    (70,170.9827)
    (80,147.4999)
  };
  \addplot coordinates { 
    (30,253.8652)
    (40,155.7141)
    (50,132.2677)
    (60,119.5409)
    (70,97.4024 )
    (80,107.9262)
  };
  \addplot coordinates { 
    (30,214.9272)
    (40,143.1439)
    (50,127.011 )
    (60,145.234 )
    (70,155.2124)
    (80,167.9822)
  };
  \addplot coordinates { 
    (30,181.8747)
    (40,150.075 )
    (50,163.0252)
    (60,131.5979)
    (70,140.8711)
    (80,156.5051)
  };
  \addplot coordinates { 
    (30,169.0475)
    (40,134.6453)
    (50,151.8079)
    (60,166.0114)
    (70,176.3179)
    (80,196.161 )
  };
  \legend{100, 200, 300, 400, 500}
  \end{axis}
\end{tikzpicture}
\caption{The filtered structure-preserving subspace iteration compared on iterations and time for different degree $N$ and number of column vectors.}
\end{figure}

\subsection{Structure-preserving filter for computing different subintervals} \label{sec:results-intervals}

\begin{table}
\caption{Description of the subintervals employed in~\cref{sec:results-intervals}, given by the indexes of the eigenvalues contained (eigs), the endpoints of the subinterval $\alpha$ and $\beta$ (rounded to 6 decimal digits), and the ratio between the size of the subinterval and the size of the positive part of the spectrum.}
\label{tab:intervals}
\begin{minipage}{\columnwidth}
\begin{center}
\begin{tabular}{l|ccc|ccc}
\cline{2-7}
& \multicolumn{3}{c|}{\cri} & \multicolumn{3}{c}{\dna} \\
\hline
eigs & $\alpha$ & $\beta$ & $(\beta-\alpha)/\lambda_{\text{max}}$ & $\alpha$ & $\beta$ & $(\beta-\alpha)/\lambda_{\text{max}}$\\
\hline
1-10   & 0        & 0.053396 & $1.42\cdot10^{-1}$ & 0        & 0.163920 & $1.12\cdot10^{-1}$ \\
11-20  & 0.053396 & 0.054153 & $2.01\cdot10^{-3}$ & 0.163920 & 0.179524 & $1.06\cdot10^{-2}$ \\
21-30  & 0.054153 & 0.056641 & $6.61\cdot10^{-3}$ & 0.179524 & 0.189453 & $6.76\cdot10^{-3}$ \\
31-40  & 0.056641 & 0.058592 & $5.18\cdot10^{-3}$ & 0.189453 & 0.201889 & $8.46\cdot10^{-3}$ \\
41-50  & 0.058592 & 0.059086 & $1.31\cdot10^{-3}$ & 0.201889 & 0.209326 & $5.06\cdot10^{-3}$ \\
51-60  & 0.059086 & 0.060611 & $4.05\cdot10^{-3}$ & 0.209326 & 0.214849 & $3.76\cdot10^{-3}$ \\
61-70  & 0.060611 & 0.060638 & $7.17\cdot10^{-5}$ & 0.214849 & 0.221373 & $4.44\cdot10^{-3}$ \\
71-80  & 0.060638 & 0.061067 & $1.14\cdot10^{-3}$ & 0.221373 & 0.228001 & $4.51\cdot10^{-3}$ \\
81-90  & 0.061067 & 0.061681 & $1.63\cdot10^{-3}$ & 0.228001 & 0.233367 & $3.65\cdot10^{-3}$ \\
91-100 & 0.061681 & 0.062225 & $1.45\cdot10^{-3}$ & 0.233367 & 0.238618 & $3.57\cdot10^{-3}$ \\
\hline
\end{tabular}
\end{center}
\end{minipage}
\end{table}

\pgfplotsset{
  width= 400,
  height= 200
}

\begin{figure}
\label{fig:intervals}
\centering
\begin{tikzpicture}[scale=0.7]
  \begin{axis}[
  title={Intervals},
  ylabel={Iterations},
  xlabel={Intervals},
  grid=major,
  xtick={10,20,30,40,50,60,70,80,90,100},
  xticklabels={1-10,11-20,21-30,31-40,41-50,51-60,61-70,71-80,81-90,91-100},
  legend pos=north west,
  legend style={draw=none, legend columns=1,font=\small,cells={anchor=west}}
  ]
  \addplot coordinates { 
    (10,  3)
    (20,  4)
    (30,  6)
    (40,  6)
    (50,  7)
    (60,  7)
    (70,  7)
    (80,  6)
    (90,  6)
    (100, 6)
  };
  \addplot coordinates { 
    (10,  4)
    (20,  7)
    (30,  20)
    (40,  23)
    (50,  51)
    (60,  173)
    (70,  82)
    (80,  87)
    (90,  72)
    (100, 129)
  };
  \addplot coordinates { 
    (10,  3 )
    (20,  4 )
    (30,  4 )
    (40,  7 )
    (50,  9 )
    (60,  10)
    (70,  14)
    (80,  18)
    (90,  18)
    (100, 22)
  };
  \legend{uniform, \cri, \dna}
  \end{axis}
\end{tikzpicture}
\caption{Filtered structure-preserving subspace iteration applied to subsequent subintervals, compared on the number of iterations for three matrices with different distribution of eigenvalues.}
\end{figure}
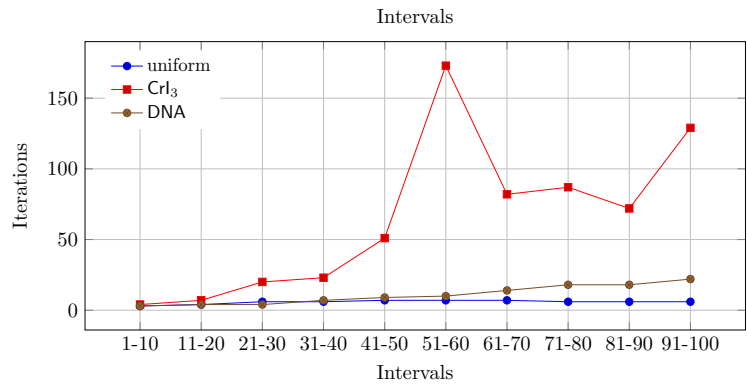

Next, we compute subsequent subintervals of ten eigenpairs for both matrices \cri and \dna, selecting the endpoints as the middle value between the largest eigenvalue contained in the previous subinterval and the smallest inside the current one (for the first subinterval we select $\alpha=0$). The degree was set to 500 and the number of column vectors to 50 in all cases. \cref{tab:intervals} shows the distribution of the subintervals and the ratio between the size of the positive subinterval and the size of the positive half of the spectrum. A narrower subinterval could require using a higher degree of the polynomial~\cite{Fang:2012:FLP}. However, in this test we have used a fixed degree for all the subintervals.

All subintervals were computed successfully, obtaining the correct number of eigenvalues, residuals below the tolerance, and $(\SignBSE,\SignOutBSE)$-orthogonality of the computed eigenvectors and biorthogonality between the left and right eigenvectors below $1\cdot10^{-13}$ in all cases.
However, the convergence is significantly worse for some subintervals, as shown in~\cref{fig:intervals}, especially in the case of the \cri matrix. We believe that the separation of the eigenvalues being reduced for subintervals further away from 0, and the appearance of clusters, that can be seen in~\cref{fig:spectrum}, has an impact on the convergence. We include, for comparison, the same test on an artificial BSE matrix with $R$ and $C$ diagonal, and evenly distributed eigenvalues. To create that matrix we have used $\lambda_{\text{max}}=2$, $\lambda_1=1$ and a separation of 0.001 between all eigenvalues. In all tests we are using the same random initial subspace for all the subintervals.

The main challenge to develop a spectrum slicing method would be selecting the subintervals in a way that improves the convergence and distributes the computational load evenly, as well as improving the properties of the filter to address these difficult cases.

\section{Concluding remarks}\label{sec:concl}

The main contribution of this article is a structure-preserving filter that, when applied to a definite BSE matrix, yields another matrix with the same properties. The two key ideas behind it are: a symmetry in the filter that results in the even coefficients being zero and employing a kernel with a positivity property, such as the Fejer or Jackson kernel.

This filter, together with the structure-preserving subspace iteration method described in this article, makes it possible to compute all the eigenvalues contained in a subinterval within the spectrum of a BSE matrix, from the endpoints of the subinterval, assuming that the number of column vectors is large enough. Maintaining the structure improves computational efficiency and saves memory, by computing explicitly only the positive eigenpairs, since the negative eigenvalues and their corresponding eigenvectors can be obtained from the former.

 The filter described here opens the door to a structure-preserving spectrum slicing method for the BSE. This is the main advantage from other methods~\cite{Napoli:2026:CAS}, that can obtain the smallest eigenvalues in magnitude, but not those in any subinterval.

Our experiments also show that the convergence is related to the degree of the filter, the size of the search space and the distribution of the spectrum. 
The next step will be to combine the work presented in this article with a method to obtain information about the distribution of the spectrum, in order to select the subintervals in a way that improves the convergence and balances the load of the computation of all the subintervals.

Lastly, while the filter was developed for the BSE application, any matrix that shows the same properties as those described for $H$ could use the proposed methods.

\begin{acknowledgements}
We would like to thank Davide Sangalli for providing us with the \yambo test matrix, and Luca Sementa for providing us with the single strand DNA helix matrix, both used in~\cref{sec:results}. We would like to thank Cl{\'e}ment Richefort for his feedback and useful comments.
\end{acknowledgements}

\appendix
\section{Inverse of a BSE matrix $H$} \label{sec:inverse-H}
\begin{lemma}\label{lma:invert}
Let $H \in \mathbb{C}^{2m\times 2m}$ be an invertible matrix with structure~\eqref{eq:bse-structure}, then $H^{-1}$ preserves the same structure.
\end{lemma}
\begin{proof}
The inverse of $H$ is
$
H^{-1}=(\SignBSE\hat{H})^{-1}=\hat{H}^{-1}\SignBSE.
$
Since $\hat{H}$ is Hermitian, $\hat{H}^{-1}$ is Hermitian as well, 
$$
\hat{H}^{-1}=
\begin{bmatrix}
B_1  &B_2\\
B_2^*&B_3
\end{bmatrix},
\qquad
H^{-1}=\begin{bmatrix}
B_1  &-B_2\\
B_2^*&-B_3
\end{bmatrix},
$$
with $B_i\in \mathbb{C}^{m\times m}$, $i=1,2,3$, and $B_1=B_1^*, B_3=B_3^*$.

To have BSE structure, $H^{-1}$ must also satisfy $B_2^T=B_2, \mybar{B}_3=B_1$.
Using the definition of inverse matrix,
\[
\begin{bmatrix}
R&C\\
-\mybar{C}&-\mybar{R}
\end{bmatrix}
\begin{bmatrix}
B_1&-B_2\\
B_2^*&-B_3
\end{bmatrix}
=
\begin{bmatrix}
I&\\
&I
\end{bmatrix}.
\]
Thus,
\[
\begin{gathered}
\begin{bmatrix}
R&C
\end{bmatrix}
\begin{bmatrix}
B_1\\
B_2^*
\end{bmatrix}
=
\begin{bmatrix}
R&C
\end{bmatrix}
\begin{bmatrix}
\mybar{B}_3\\
\mybar{B}_2
\end{bmatrix}
=
I
,\\
\begin{bmatrix}
-\mybar{C}&-\mybar{R}
\end{bmatrix}
\begin{bmatrix}
B_1\\
B_2^*
\end{bmatrix}
=
\begin{bmatrix}
-\mybar{C}&-\mybar{R}
\end{bmatrix}
\begin{bmatrix}
\mybar{B}_3\\
\mybar{B}_2
\end{bmatrix}
=
0.
\end{gathered}
\]
Therefore,
\[
\begin{bmatrix}
R&C\\
-\mybar{C}&-\mybar{R}
\end{bmatrix}
\begin{bmatrix}
B_1-\mybar{B}_3\\
B_2^*-\mybar{B}_2
\end{bmatrix}
=
0.
\]
Since $H$ is invertible, it follows that $B_1=\mybar{B}_3$, $B_2=B_2^T$.
Therefore $H^{-1}$ has BSE structure.
\end{proof}

\end{document}